\documentclass[a4paper,11pt]{article}
\usepackage{amssymb}
\usepackage{amsfonts}
\usepackage{amsmath}
\usepackage{color}
\usepackage[affil-it]{authblk}

\usepackage{graphicx}
\usepackage[all]{xy}
\usepackage{array}
\usepackage{youngtab}
\let \: \colon
\newcommand{\mathbold}{\mathbb}

\newcommand{\bZ}{{\mathbold Z}}

\newcommand{\und}{\underline}

\newcommand{\tn}{\otimes}           % Tensor

\newcommand{\sO}{{\mathcal{O}}}
\newcommand{\cB}{{\mathcal{B}}}

\newcommand{\Mod}[1]{\mathsf{mod}({{#1}})}
\newcommand{\CoMod}[1]{\mathsf{comod}({{#1}})}

\newcommand{\Hom}{\mathrm{Hom}}
\newcommand{\Ind}{\mathrm{Ind}}

\newcommand{\End}{{\mathrm{End}}}

\newcommand{\rightlabelxyarrows}[2]{{\ar@<0.7ex>^-{#1}[r]\ar@<-0.7ex>_-{#2}[r]}}

\newcommand{\displaylabelfork}[6]{{     \entrymodifiers={+!!<0pt,\fontdimen22\textfont2>}
        \def\objectstyle{\displaystyle}
\xymatrix{{#1} \ar^-{#2}[r] & {#3} \ar@<0.7ex>^-{#4}[r]\ar@<-0.7ex>_-{#5}[r] & {#6}}}}

\newcommand{\predisplaylabelfork}[6]{{{#1} \ar^-{#2}[r] & {#3} \ar@<0.7ex>^-{#4}[r]\ar@<-0.7ex>_-{#5}[r] & {#6}}}

\newcommand{\Vect}{\mathcal{V}}
\newcommand{\V}{\mathcal{V}}
\newcommand{\fS}{\mathfrak{S}}

\newcommand{\sC}{\mathcal{C}}
\newcommand{\cY}{\mathcal{Y}}
\newcommand{\cH}{\mathcal{H}}
\newcommand{\PP}{\mathcal{P}}
\newcommand\kk{\Bbbk}

\newcommand{\bw}{\bigwedge}
\newcommand{\la}{\lambda}

\usepackage{amsthm}

\newtheorem{theorem}{Theorem}[section]

\newtheorem{lemma}[theorem]{Lemma}
\newtheorem{proposition}[theorem]{Proposition}
\newtheorem{corollary}[theorem]{Corollary}

\theoremstyle{definition}
\newtheorem{definition}[theorem]{Definition}

\newtheorem{remark}[theorem]{Remark}

\newtheorem*{acknowledgement}{Acknowledgement}

\title{Quantum Polynomial Functors }
\author{Jiuzu Hong%
  \thanks{\texttt{jiuzu@email.unc.edu}}}
\affil{Mathematics Department,\\ UNC Chapel Hill }

\author{Oded Yacobi%
  \thanks{\texttt{oded.yacobi@sydney.edu.au}}}
\affil{School of Mathematics and Statistics,\\ University of Sydney}

\date{\today}

\begin{document}

%\begin{classification}
%Primary 13F35, 13K05;
%Secondary 16Y60, 05E05, 05A17.
%\end{classification}

%\begin{keywords}
%
%\end{keywords}

%\date{\today. \printtime}

\maketitle

\abstract 
\text{  }We construct a category of quantum polynomial functors which deforms Friedlander and Suslin's category of strict polynomial functors.  The main aim of this paper is to develop from first principles the basic structural properties of this category (duality, projective generators, braiding etc.) in analogy with classical strict polynomial functors.  We then apply the work of Hashimoto and Hayashi in this context to construct quantum Schur/Weyl functors, and use this to provide new and easy derivations of quantum $({\rm GL}_m,{\rm GL}_n)$ duality, along with other results in quantum invariant theory.

\section{Introduction}

%The category $\PP$ of strict polynomial functors was introduced by Friedlander and Suslin in their study of the cohomology of finite group schemes \cite{FS}.  In this work, we define a new category $\PP_q$ of quantum polynomial functors, which deforms Friedlander and Suslin's category.    The aim of this paper is to develop some of the basic properties of this category, analogous to those of $\PP$.  

Let $\kk$ be a commutative ring and choose $q\in\kk^\times$.  The category $\PP_q^d$ of  quantum polynomial functors of homogeneous degree $d$ consists of functors $\Gamma_q^d\V\to\V$, where $\V$ is the category of finite projective $\kk$-modules, and $\Gamma_q^d\V$ is the category with objects natural numbers and morphisms given by
$$
\Hom_{\Gamma_q^d\V}(m,n):=\Hom_{\cB_d}(V_m^{\tn d},V_n^{\tn d}).
$$      
Here $\cB_d$ is the Artin braid group, $V_m$ denotes the free $\kk$-module of rank $n$ and the action of $\cB_d$ on $V_m^{\tn d}$ is given in Section \ref{sec-qmatrices}.  We think of the category $\Gamma_q^d\V$ as the category of standard Yang-Baxter spaces $(V_n,R_n)$ (Section \ref{sec-qmatrices}), and the morphisms can be viewed as degree $d$ regular functions on the quantum Hom-space between standard Yang-Baxter spaces (although, as is usual in quantum algebra, only the regular functions are actually defined).

The purpose of this paper is to develop the basic structure theory of the category $\PP_q^d$ in analogy with Friendlander and Suslin's work \cite{FS}.
%
%When $q=1$ there is an equivalence between $\PP_1^d$ and the category of  polynomial functors of degree $d$ as defined in \cite{FS}.  There is also a useful alternative  characterization of objects in $\PP_q^d$ as sequences of vector spaces endowed with compatible coactions of the algebra of quantum $m\times n$ matrix space (cf. Proposition \ref{prop-char}).
%
We first need to  develop a theory of  quantum linear algebra in great generality using Yang-Baxter spaces, and this is undertaken in Section \ref{Qlinalg}.  

We define morphisms between Yang-Baxter spaces over an algebra, and provide  a universal  characterization of quantum Hom-space algebra in Lemma \ref{universal_q_operator}. From this formalism we can derive many results about quantum Hom-space algebras functorialy.  In particular, the dual of the Hom-space between two Yang-Baxter spaces of degree $d$ is identified with certain braid group intertwiners, generalizing a well-known description of $q$-Schur algebras (Proposition \ref{Natural_Iso}).  We construct the algebra of quantum $m\times n$  matrix space by specializing this theory to standard Yang-Baxter spaces.

We note that when the Yang-Baxter spaces are equal the quantum Hom space algebras  have appeared in \cite{Ph},\cite{HH}, but in the generality studied here these are new.    
We further remark  that the general formalism of quantum linear algebra we develop builds  on the work of Hashimoto-Hayashi \cite{HH}, but it is not the same. They only consider the quantum Hom-space algebra between the same Yang-Baxter spaces, whereas for us it is crucial to build in morphisms between different Yang-Baxter spaces.

After the basic of quantum multilinear algebra are in place, we set out to develop the theory of quantum polynomial functors.  To begin, the functor $\Gamma_q^{d,m}:\Gamma_q^d\V\to\V$ given by $n \mapsto \Hom_{\cB_d}(V_m^{\otimes d}, V_n^{\otimes d})$ is called the quantum divided power functor.  Theorem \ref{Rep_thm} states that $\Gamma_q^{d,m}$ is a projective generator of $\PP_q^d$ when $m\geq d$.  
This uses a finite generation property for quantum polynomial functors, which we prove in Proposition \ref{prop-gen}.

Theorem \ref{Rep_thm} has several corollaries.  It implies for instance that when $n \geq d$ we have an equivalence $$\PP_q^d \cong \Mod{S_q(n,d)},$$ between the category of quantum polynomial functors of degree $d$ and the category of modules over the q-Schur algebra $S_q(n,d)$ that are finite projective over $\kk$.  It also allows us to construct functors which represent weight spaces for representation of the quantum general linear group (Corollary \ref{Weight_Rep}).  We note that from this and Proposition \ref{Natural_Iso} one can immediately deduce the double centralizer property of Jimbo-Schur Weyl duality ( Corollary \ref{Double_centrallizer}).  

In fact Theorem \ref{Rep_thm} is also needed to show that the R-matrix of the quantum general linear group are suitably functorial, that is they are natural with respect to morphisms in $\Gamma_q^d\V$, and thereby define a braiding on $\PP_q:=\bigoplus_{d=0}^\infty \PP_q^d$ (Theorem \ref{Braiding}).  We emphasize that these results are elementary consequences of the \textit{definition} of quantum polynomial functors.

In Section \ref{sec: schurfunctors} we use the extensive work of Hashimoto and Hayashi \cite{HH} to define quantum Weyl/Schur functors.  We show that these objects are dual to each other in Theorem \ref{SW-functor}, and use them to describe the simple objects in $\PP_q$.   Finally in Section \ref{sec-invariants} we specialize to the case when $\kk$ is a field of characteristic zero and $q$ is generic,  and give new and simplified proofs of the invariant theory of quantum ${\rm GL}_n$.  We use Theorem \ref{Rep_thm} to give an easy proof of the duality between quantum ${\rm GL}_m$ and  ${\rm GL}_n$.  We also formulate and derive the equivalence of this duality to the quantum first fundamental theorem and Jimbo-Schur-Weyl duality.  

We remark that  quantum $({\rm GL}_m,{\rm GL}_n)$-duality is due to  Zhang \cite{Zh} and  Ph\'ung  \cite{Ph}.  (Zhang also derives Jimbo-Schur-Weyl duality from $({\rm GL}_m,{\rm GL}_n)$-duality.)  The quantum FFT that we prove first appears in \cite{GLR} with a much more complicated proof. (Other versions of the quantum FFT appear in \cite{Ph} and  \cite{LZZ}.)  We remark also that our approach to quantum invariant theory applies to the other settings where a theory of strict polynomial functors has been constructed (cf. Remark \ref{LastRemark}).  

Finally, an important  problem concerning $\PP_q$ remains open: to define composition of quantum polynomial functors.  In Section \ref{sec:final} we discuss obstructions to defining composition in our setting, and speculate on possible generalizations of our constructions that would allow for composition, and thus provide the sought-after quantum plethysm.  We hope this paper is a significant step in this program.      

This work is inspired by our  previous works \cite{HY1,HY2,HTY} on polynoimal functors and categorifications.

\begin{acknowledgement}
We thank Joseph Bernstein for his encouragement and important discussions during the initial stages of this project, and also  Roger Howe and Andrew Mathas for helpful conversations.  We also thank Antoine Touz\'e for introducing us to polynomial functors.  
The authors would also like to thank the anonymous  referees for many helpful suggestions that greatly aided the clarity of this paper.
The second author is supported by an Australian Research Council DECRA fellowship.  
\end{acknowledgement}

\section{Quantum matrix spaces}
The  theory of quantum $n\times n$ matrix space is well-known and highly developed  (cf.  \cite{DJ2}, \cite{PW}, \cite{Ma}). In order to develop a theory of quantum polynomial functors, we need to generalize this theory, namely we study the quantum $m\times n$-matrix space where $m$ is not necessarily equal to $n$.  

We first develop a version of quantum linear algebra in even greater generality using general Yang-Baxter spaces (see below).
We introduce the notion of morphisms of Yang-Baxter spaces over algebras, and study their compositions and quantum Hom-spaces.   
The quantum matrix space of interest are then  special cases of these more general quantum algebras, and algebraic structures such as products and coproducts are easily obtained  from the general constructions.
\label{Qlinalg}
\subsection{Quantum linear algebras}
Let $\kk$ be a commutative ring. For any two $\kk$-modules $V,W$, throughout this paper  $\Hom(V,W)$ denotes $\Hom_\kk(V,W)$ for brevity,  and similarly $V\otimes W$ denotes the tensor product $V\otimes_\kk W$.  For any $\kk$-module $V$, $V^*$ denotes the dual space $\Hom_\kk(V,\kk)$.

Let $\cB_d$ be the Artin braid group: it is generated by $T_1,T_2,\cdots, T_{d-1}$ subject to the  relations
\begin{equation}
\label{eq:braid}
\begin{aligned}
T_iT_{j}&=T_jT_i \qquad \text{ if } |i-j|>1, \\
T_{i}T_{i+1}T_i&=T_{i+1}T_iT_{i+1} 
\end{aligned}
\end{equation}
Let $\fS_d$ denote the symmetric group on $d$ letters.  For any $w\in \fS_d$ we define $T_w\in\cB_d$ by choosing a reduced expression $w=s_{i_1}\cdots s_{i_\ell}$ and setting $T_w=T_{i_1}\cdots T_{i_\ell}$.

For any free $\kk$-module $V$ of finite rank,  a \textbf{Yang-Baxter operator} is a $\kk$-linear operator
$R:V^{\otimes 2}\to V^{\otimes 2}$ such that
$R$ satisfies the Yang-Baxter equation, i.e. the following equation holds in $\End(V^{\otimes 3})$:
$$R_{12}R_{23}R_{12}=R_{23}R_{12}R_{23}, $$
where $R_{12}=R\otimes 1_V$ and $R_{23}=1_V\otimes R$. 

Such a pair $(V,R)$ is called a \textbf{Yang-Baxter space}.   To $(V,R)$ we associate the right representation $\rho_{d,V}:\cB_d \to \End(V^{\otimes d})$ via the formula
$$
T_i \mapsto 1_{V^{\otimes i}}\otimes R \otimes 1_{V^{\otimes^{d-i-1}}}.
$$
Often we suppress $R$ in the notation and refer to a free $\kk$-module $V$ as a ``Yang-Baxter space''.  In this case, the operator $R$ is implicit and when necessary is denoted $R_V$.   For now the operator $R$ is quite general, in Section \ref{sec-qmatrices} we will specialize to a specific (standard) set of $R$-matrices.  

Now consider two Yang-Baxter spacess $V,W$.   Let $T(V,W)$ be the tensor algebra of $\Hom(V,W)$, which is graded 
$$T(V,W)=\bigoplus_{d\geq 0} T(V,W)_{d} ,$$
where 
$$T(V,W)_d:=\Hom(V,W)^{\otimes d}\simeq  \Hom(V^{\otimes d}, W^{\otimes d}). $$
Let $I(V,W)$ be the two sided ideal generated by 
$$R(V,W):=\{X \circ  R_V-R_W\circ X\text{ }| \text{ } X\in \Hom(V^{\otimes 2},W^{\otimes 2})\}.  $$
 The ideal $I(V,W)$ is homogeneous 
 $$I(V,W)=\bigoplus_{d\geq 0}I(V,W)_d, $$
 where $I(V,W)_d$ is spanned by 
 $$\Hom(V,W)^{\otimes i-1}\otimes R(V,W)\otimes \Hom(V,W)^{\otimes d-i-1}$$
for $i=1,2,..., d-1$.  We define 
 \begin{equation}
 \label{quantum_algebra}
 A(V,W): =T(V,W)/I(V,W) .
 \end{equation}
 The algebra $A(V,W)$ is called the \textbf{quantum Hom-space algebra} from $W$ to $V$ (cf. \cite[\S 3]{Ph} and \cite[\S 3]{HH}).  
 \begin{remark}
 While it may seem confusing to denote by $A(V,W)$ the morphisms \textit{from} $W$ \textit{to} $V$, this is inherent to the quantum point-of-view.  One should think of $A(V,W)$ as the ring of regular functions on the space of morphisms from $W$ to $V$  (even though - as is typical in the quantum setting - the latter is not defined).  Classically, i.e. when the Yang-Baxer operators are just the flip maps, $A(V,W)$ equals $S(\Hom(V,W))$, which is of course isomorphic to the regular functions on $\Hom(W,V)$.  
 \end{remark}
 
$A(V,W)$ has a natural grading
$$A(V,W)=\bigoplus_{d\geq 0} A(V,W)_d ,$$
where 
\begin{equation}
\label{q_Hom_coin}
A(V,W)_d=  T(V,W)_d/ I(V,W)_d.
\end{equation}

Let $C$ be a $\kk$-algebra with multiplication $m: C\times C\to C$. We introduce the following notion.
\begin{definition}
A \textbf{Yang-Baxter morphism}  from $(V,R_V)$ to $(W,R_W)$ over $C$ is a   
 $\kk$-linear map $P:V\to W\otimes C$ such that the following diagram commutes:
\begin{equation}
\xymatrix{
V^{\otimes 2} \ar[d]^{R_V}\ar[r]^>>>>>{P^{(2)}} & W^{\otimes 2}\otimes C
\ar[d]^{R_W\otimes 1_C}\\
V^{\otimes 2} \ar[r]^>>>>>{P^{(2)}} & W^{\otimes 2}\otimes C
}
\end{equation}
Here $P^{(2)}$ is the  composition:
$$ V^{\otimes 2}\xrightarrow{P\otimes P} W\otimes C\otimes W\otimes C \xrightarrow{\text{flip}} W^{\otimes 2}\otimes C\otimes C \xrightarrow{1\otimes m} W^{\otimes 2}\otimes C.$$
\end{definition}

Let $\{v_i\}$ (resp. $\{w_j\}$) be a basis of $V$ (resp. $W$).  With this choice of basis, we can write the operator $R_V$  in terms of a matrix $(R_{V,ij}^{k\ell})$,
\begin{equation}
\label{R-matrix}
v_i\otimes v_j\mapsto  \sum_{k \ell}  R_{V, ij}^{k\ell} w_k\otimes w_\ell.
\end{equation}
Similarly we can express $R_W$ in terms of the matrix  $(R_{W, ij}^{k\ell})$. 

The following lemma is immediate from the definition of a Yang-Baxter morphism.
\begin{lemma}
\label{quadratic_relations}
Given any $\kk$-linear map $P:  V\to W\otimes C$, with 
$$P(v_i)=\sum_j  w_j\otimes P_{ji}, $$
for any $i,j$, and $P_{ji}\in C$.  The map $P$ is a Yang-Baxter morphism over $C$ if and only if for any $i,j,p,q$, the following quadratic relation holds
\begin{equation}
\label{Quadratic_relations}
\sum_{k,\ell}  R^{pq}_{W, k\ell }   P_{ki} P_{\ell j} =\sum_{k,\ell}  R_{V, ij}^{k\ell}  P_{pk} P_{q\ell} . 
\end{equation}
\end{lemma}

 Let $\delta_{V,W}:V\to W\otimes \Hom(W,V)$ be the canonical map induced from the identity map $\Hom(W,V)\to \Hom(W,V)$.  
We can precisely describe it: Let $\{v_i\}$ be a basis of $V$ and $\{w_j\}$ be a basis of $W$, then $\delta_{V,W}$ is given by 
$$ v_i\mapsto \sum_{j}  w_j\otimes \phi_{ji},$$
for any $i$, where $\phi_{ji}: W\to V$ is the map 
\begin{equation}
\label{phi_function}
\phi_{ji}(w_k)= \begin{cases} v_i   \quad \text{ if } k=j \\  0   \quad \text{ otherwise } .  \end{cases}
\end{equation}
It is easy to check that  $\delta_{V,W}$ doesn't depend on the choice of bases.   

The map $\delta_{V,W}$ in further  induces a $\kk$-linear operator 
\begin{equation}
\label{Universal_YB_morphism}
\delta_{V,W}:V\to W\otimes A(W,V),
\end{equation}
since $ A(W,V)_1=\Hom(W,V)$ is a $\kk$-submodule of $A(W,V)$.    
  \begin{lemma}
  \label{universal_q_operator}
The  map $\delta_{V,W}: V\to W\otimes A(W,V)$  is a Yang-Baxter morphism over $A(W,V)$.
\end{lemma}
\begin{proof}
Let $\{v_i\}$ (resp. $\{w_j\}$) be a basis of $V$ (resp. $W$).  By the construction of $\delta_{V,W}$, the map $\delta_{V,W}^{(2)}:  V^{\otimes 2} \to W^{\otimes 2}\otimes A(W,V)$ is given by 
$$v_i\otimes v_j\mapsto   \sum_{k,\ell}  w_k\otimes w_\ell  \otimes   \phi_{ki} \phi_{\ell j},$$
for any $i,j$.  
By Lemma \ref{quadratic_relations}, we need to check that $\{\phi_{ki}\}$ satisfies the following quadratic relations:
$$\sum_{k,\ell}  R^{pq}_{W, k\ell }   \phi_{ki} \phi_{\ell j} =\sum_{k,\ell}  R_{V, ij}^{k\ell}  \phi_{pk}\phi_{q\ell} ,$$
for any $i,j,p,q$. 
These are  exactly the quadratic relations  defining the algebra $A(W,V)$ in (\ref{quantum_algebra}).
\end{proof}

The following lemma shows that the quantum Hom-space algebra is characterized by a universal property.  In the case where $V=W$ are the same Yang-Baxter space this lemma follows from Theorem 3.2 in \cite{HH}. 
\begin{lemma}
\label{Uni_q_op}
Let $(V,R_V),(W,R_W)$ be two Yang-Baxter spaces.  Then the map $\delta_{V,W}: V\to W\otimes A(W,V)$ is the unique Yang-Baxter morphism  such that for any Yang-Baxter morphism $P:V\to W\otimes C$ over a $\kk$-algebra $C$,   there exists a unique morphism of algebras $\tilde P: A(W,V)\to C$ such that the following diagram commutes:
\begin{equation}
\label{q_Hom_algebra}
\xymatrix{
V \ar[r]^<<<<<{\delta_{V,W}} \ar[dr]_{P} & W\otimes A(W,V) \ar[d]^{1_W \otimes \tilde P}\\
& W\otimes C
}
\end{equation}
\end{lemma}
\begin{proof}

Assume  the map $P: V\to W\otimes C$ is given by 
$$v_i\mapsto \sum_j  w_j\otimes P_{ji}, $$
for any $i,j$, and $P_{ij}\in C$. 
Then $(P_{ji})$ satisfies the quadratic relations in (\ref{Quadratic_relations}).     We define a map $\tilde{P}: \Hom(W,V)\to C$ such that $\tilde{P}(\phi_{ji})=P_{ji} $.  Then this map uniquely extends to a homomorphism of algebras $\tilde{P}:A(W,V)\to C$, since $(\phi_{ji})$ and $(P_{ji})$ satisfies the same quadratic relations.      Clearly we have the commutative diagram (\ref{q_Hom_algebra}).
The uniqueness of $\delta_{V,W}$ is clear.
\end{proof}

Given three Yang-Baxter spaces $V,W,U$ and Yang-Baxter morphisms  $P: V\to W\otimes C$ and $Q: W\to U\otimes D$ over algebras $C$ and $D$ respectively, we denote by $Q\circ P: V\to U\otimes D\otimes  C$ the composition of $P$ and $Q$.
\begin{lemma}
\label{Com_q_op}
The composition $Q\circ P$ is  a Yang-Baxter morphism from $(V,R_V)$ to $(U,R_U)$ over $D\otimes C$.
\end{lemma}
\begin{proof}
We choose a basis $\{v_i\}$ in $V$,  $\{w_j\}$ in $W$, and $\{u_k\}$ in $U$.  Assume $P$ can be represented by the matrix $(P_{ij})$ with respect to the basis $\{v_i\}$ and $\{w_j\}$, and $Q$ can be represented by the matrix $(Q_{jk})$ with respect to $\{w_j\}$ and $\{u_k\}$. 
Then the composition $Q\circ P: V\to U\otimes D\otimes C$ can be represented by the matrix $(\sum_{s} Q_{js} \otimes P_{si})$.  
By Lemma \ref{quadratic_relations}, it is enough to check that 
\begin{equation}
\label{composition}
\sum_{k,\ell}  R^{pq}_{U, k \ell} (\sum_s Q_{ks}\otimes P_{si} ) (\sum_{t} Q_{\ell t}\otimes P_{t j})= \sum_{k,\ell}  R^{k \ell}_{V, ij}  (\sum_s Q_{ps}\otimes P_{sk} ) (\sum_t  Q_{qt}\otimes P_{t\ell}) ,
\end{equation}
for any $i,j,p,q$.

Note that the left hand side is equal to 
\begin{equation}
\label{Term_1}
\sum_{s,t} (\sum_{k,\ell}  R^{pq}_{U, k\ell} Q_{ks}Q_{\ell t})\otimes P_{si}P_{tj}   
\end{equation}
Since $Q$ is a Yang-Baxter morphism over $D$,  (\ref{Term_1}) is equal to 
\begin{equation}
\label{Term_2}
\sum_{s,t} ( \sum_{k,\ell}  R^{k\ell}_{W,st} Q_{pk}Q_{q\ell}) \otimes P_{si}P_{tj}.  
\end{equation}
By changing the order of summations, (\ref{Term_2}) is equal to 
\begin{equation}
\label{Term_3}
\sum_{k,\ell}  Q_{pk}Q_{q\ell} \otimes  (  \sum_{s,t}  R^{k\ell}_{W,st} P_{si}P_{tj} ) .
\end{equation}
Since $P$ is  a Yang-Baxter morphism over $C$,  (\ref{Term_3}) is equal to 
\begin{equation}
\label{Term_4}
\sum_{k,\ell}  Q_{pk}Q_{q\ell} \otimes  (\sum_{s,t}  R^{st}_{V,ij} P_{ks}P_{\ell t} ). 
\end{equation}
By switching indices $k, \ell$ and $s, t$, (\ref{Term_4}) is exactly the right hand side of (\ref{composition}).
\end{proof}
By Lemma \ref{Com_q_op}, the operator $\delta_{W,V}\circ\delta_{U,W}$ is a Yang-Baxter morphism over $$A(V,W)\otimes A(W,U).$$
Applying Lemma \ref{Uni_q_op} to  $\delta_{W,V}\circ\delta_{U,W}$ we obtain a morphism of algebras
\begin{equation}
\label{co-mult}
\Delta_{VWU}: A(V,U)\to A(V,W)\otimes A(W,U).  
\end{equation}
It preserves degree, i.e. for each $d\geq 0$, we have 
$$ \Delta_{VWU}: A(V,U)_d\to A(V,W)_d\otimes A(W,U)_d. $$
By the universal property of Yang-Baxter morphisms, $\Delta_{*,*,*}$ satisfies co-associativity, i.e. for any Yang-Baxter spaces $V,W,U,Z$, we have 
\begin{equation}
\label{co-associativity}
(1\otimes \Delta_{WUZ})\circ \Delta_{VWZ}=(\Delta_{VWU}\otimes 1)\circ \Delta_{VUZ} .
\end{equation}
Let 
\begin{equation}
\label{Schur}
S(V,W;d):=(A(W,V)_d)^* .
\end{equation}
This space, a kind of ``rectangular generalized Schur algebra'', will be used throughout this paper.  We will see below (Section \ref{sec-qmatrices}) that when $V=W$ are the standard Yang-Baxter spaces, then $S(V,V;d)$ is the $q$-Schur algebra.

From $\Delta_{V,W,U}$, we obtain by duality a $\kk$-bilinear map 
\begin{equation}
\label{Schur_bilinear}
m_{UWV}: S(W,V;d)\times S(U,W;d)\to S(U,V;d). 
\end{equation}
For any $a\in S(W,V;d)$ and $b\in S(U,W;d)$, we denote by $b\circ a$ the element $m_{UWV}(a,b)\in S(U,V;d)$.  It is given by the following composition 
\begin{equation}
\xymatrix{
A(V,U)_d\ar[r]^<<<<<{\Delta_{VWU}}\ar[dr]^{b\circ a} & A(V,W)_d\otimes A(W,U)_d\ar[d]^{a\otimes b}\\
& \kk
}.
\end{equation}

By co-associativity of $\Delta_{*,*,*}$,  we naturally have associativity of $m_{*,*,*}$, i.e.  for any Yang-Baxter spaces $V,W,U,Z$,
\begin{equation}
\label{associativity}
m_{ZWV}\circ (1\times m_{ZUW})=m_{VWU}\circ (m_{UWV} \times 1 ) .
\end{equation}

The following proposition generalizes \cite[Theorem 11.3.1]{PW} to the case where $V\neq W$ and the Hecke algebra is replaced by the braid group.  (Note that this proof is simpler; in particular no dimension arguments are used and hence we don't need to produce a basis for $S(V,W;d) $.)     
\begin{proposition}
\label{Natural_Iso}
Let $V,W$ be Yang-Baxter spaces.  Then there exists a natural isomorphism 
$$S(V,W;d)\simeq \Hom_{\cB_d}(V^{\otimes d},W^{\otimes d}). $$
\end{proposition}

\begin{proof}
We define a representation of $\cB_d$  on $\Hom(V^{\otimes d},W^{\otimes
d})$, where for each $i$, $T_i$ is the following opeartor 
$$
X\mapsto X\circ \rho_{d,V}(T_i)^{-1}-\rho_{d,W}(T_i)^{-1}\circ X, 
$$
for $X\in \Hom(V^{\otimes d},W^{\otimes d})$.  Recall that $\rho_{d,V}$ (resp. $\rho_{d,W}$) denotes the right action of $\cB_d$ on $V^{\otimes d}$ (resp. $W^{\otimes d}$).  Note that
$
\Hom_{\cB_d}(V^{\otimes d},W^{\otimes d})$ is the just the invariant space  $\Hom(V^{\otimes d},W^{\otimes d})^{\cB_d}$ of $\cB_d$ on $\Hom(V^{\otimes d},W^{\otimes d})$.

Similarly we define a representation of  $\cB_d$ on $\Hom(W^{\otimes d},V^{\otimes
d})$, where for each $i$, $T_i$ is the following operator
$$
Y\mapsto Y \circ \rho_{d,W}(T_i)-\rho_{d,V}(T_i)\circ Y, 
$$
for $Y\in \Hom(W^{\otimes d},V^{\otimes d})$.  Note that from (\ref{q_Hom_coin}) we have that 
$A(W,V)_d$ is the coinvariant space  $\Hom(W^{\otimes d},V^{\otimes d})_{\cB_d}$ of $\cB_d$ on  $\Hom(W^{\otimes d},V^{\otimes d})$. 

Now consider the following perfect  non-degenerate pairing 
$$\langle\cdot,\cdot\rangle:\Hom_\kk(V^{\otimes d},W^{\otimes d})\times \Hom(W^{\otimes d},V^{\otimes d})\to \kk , $$
given by $\langle X,Y\rangle:=\rm{trace}(Y\circ X) $. 
It is clear that the representation of $\cB_d$ on $\Hom(W^{\otimes d},V^{\otimes d})$ is contragradient  to the representation of $\cB_d$ on $\Hom(V^{\otimes d},W^{\otimes d})$ with respect to the above non-degenerate pairing.  Therefore we have a natural isomorphism 
$$(\Hom(W^{\otimes d},V^{\otimes d})_{\cB_d})^* \simeq \Hom(V^{\otimes d},W^{\otimes d})^{\cB_d} ,$$
i.e.  there exist a natural isomorphism 
$$S(V,W;d)\simeq \Hom_{\cB_d}(V^{\otimes d},W^{\otimes d}). $$

\end{proof}

Let $\Delta: \Hom(V,U)\to \Hom_\kk(V,W)\otimes \Hom_\kk(W,U)$ be the map 
$$\Delta(\phi_{ji})=\sum_s  \phi_{js}\otimes \phi_{si} ,$$
where $\phi_{ji}\in\Hom(V,U)$, $\phi_{si}\in \Hom(V,W)$ and $\phi_{js}\in \Hom(W,U)$ are defined as in \ref{phi_function} after a choice of bases for $V,W,U$. 
The map $\Delta$ induces a map of tensor algebras $\Delta:  T(V,U)\to T(V,W)\otimes T(W,U)$.  

\begin{proposition}
\label{Iso_Com}
Given three Yang-Baxter spaces $V,W,U$, then the following diagram commutes:
\begin{equation}
\label{Schur_commutativity}
\xymatrix{
S(W,V;d)\otimes S(U,W;d)\ar[r]\ar[d] & \Hom_{\cB_d}(W^{\otimes d},V^{\otimes d})\otimes \Hom_{\cB_d}(U^{\otimes d},W^{\otimes d}) \ar[d]\\
S(U,V;d) \ar[r] &\Hom_{\cB_d}(U^{\otimes d},V^{\otimes d})
}.
\end{equation}
\end{proposition}
\begin{proof}

Recall that $A(V,U)$ is a quotient of $T(V,U)$ by the  relations,
$$\sum_{k \ell} (R_{U,k\ell}^{p q} \phi_{k  i}\otimes \phi_{\ell j}-R^{k\ell}_{V, ij}\phi_{pk}\otimes \phi_{q\ell} ),$$
for any appropriate indices $i,j,p,q$ after a choice of bases of $V,W,U$.  It is a tedious but straight forward to show that the defining quadratic relations for $A(V,U)$ will be sent to zero by the composition map $\pi\otimes \pi\circ \Delta: T(V,U)\xrightarrow{\Delta} T(V,W)\otimes T(W,U)\xrightarrow{\pi\otimes \pi} A(V,W)\otimes A(W,U)$, where $\pi$ is the projection map. It means  that we have the following commutative diagram 
$$\xymatrix{
A(V,W)\otimes A(W,U) & T(V,W)  \otimes T(W,U)     \ar[l]_<<<{\pi\otimes \pi}\\
A(V,U)\ar[u]_{\Delta_{V,W,U}} &T(V,U)\ar[l]_{\pi} \ar[u]^{\Delta}
}.$$
   Moreover note that these maps preserve degrees.    After taking the dual on each degree and applying Proposition \ref{Natural_Iso}, the commutativity of (\ref{Schur_commutativity}) follows.
\end{proof}

\subsection{Quantum matrix space}
\label{sec-qmatrices}
We fix a commutative ring $\kk$ and an element $q \in \kk^\times$.  
Let $\cH_d$ be the Iwahori-Hecke algebra of type A: it is the  $\kk$-algebra
generated by $T_1,T_2,...,T_{d-1}$ subject to the relations:
\begin{equation}
\label{eq:hecke}
\begin{aligned}
T_iT_{j}&=T_jT_i \qquad \text{ if } |i-j|>1, \\
T_{i}T_{i+1}T_i&=T_{i+1}T_iT_{i+1} \\
(T_i-q)(T_i+q^{-1}) &=0.
\end{aligned}
\end{equation}
The algebra $\cH_d$ is a quotient of the group algebra of the braid group $B_d$, by the third relation above which we call the ``Hecke relation''.

Let $(V_n,R_n)$ be the \textbf{standard Yang-Baxter space}, where $V_n=\kk^n$ with  basis $e_1,e_2,\cdots,e_n$, and $R_n:V_n\otimes V_n\to V_n\otimes V_n$ is the $\kk$-linear operator defined
by:
\begin{equation}
\label{R_matrix}
R_n(e_i\otimes e_j)=
\begin{cases}
  e_j\otimes e_i  \qquad  \text{ if } i< j\\
  qe_i\otimes e_j \qquad \text{ if } i=j\\
 (q-q^{-1})e_i\otimes e_j + e_{j}\otimes e_i \qquad \text{ if } i>j
\end{cases},
\end{equation}
where $q\in \kk$.
The following is well-known and easy to check (see e.g. Lemma 4.8 in \cite{T}).
\begin{lemma}
For any $n$,   $R_n:V_n^{\otimes 2}\to V_n^{\otimes 2} $ is a Yang-Baxter operator. Moreover, $R_n$ satisfies the Hecke relation in (\ref{eq:hecke}), i.e. 
$$(R_n-q)(R_n+q^{-1})=0 .$$ 
\end{lemma}
Let $\rho_{d,n}:\cH_d\to \End(V_n^{\otimes d})$ denote the corresponding right $\cH_d$-module.
Recall the map  $\delta_{V_n,V_m}: V_n\to V_m\otimes A_q(m,n)$  given in (\ref{Universal_YB_morphism}), we can write 
$$\delta_{V_n,V_m}(e_i)=\sum_{j} e_j\otimes x_{ji} , $$
where $\{x_{ji}\}$ is the standard basis of $\Hom(V_m,V_n)$ mapping $e_k \mapsto \delta_{ik}e_j$ and $\delta_{ik}$ is the Kronecker symbol.  
%such that 
%$$\delta_{n,m}(e_i)=\sum e_j\otimes x_{ji}\in V_m\otimes \Hom(V_m,V_n), $$
%where $e_i\in V_n$. 

\begin{lemma}
\label{Qpresent}
The algebra $A(V_m,V_n)$  is  generated by $x_{ji}$, $1\leq j\leq
m$, $1\leq i\leq n$, subject to the following relations:
\begin{align*}
 k>\ell \Rightarrow x_{ik}x_{i\ell} &=qx_{i\ell}x_{ik}\\
i>j \Rightarrow  x_{ik}x_{jk} &=qx_{jk}x_{ik}\\
k>\ell \text{ and } i>j \Rightarrow x_{i\ell}x_{jk} &=x_{jk}x_{i\ell} \\ 
k>\ell \text{ and } i>j \Rightarrow x_{ik}x_{j\ell}-x_{j\ell}x_{ik} &=(q-q^{-1})x_{i\ell}x_{jk}.
\end{align*}
\end{lemma}
\begin{proof}
By Lemma \ref{universal_q_operator},  $\delta_{V_n,V_m}$ is a Yang-Baxter morphism from $(V_n,R_n)$ to $(V_m,R_m)$ over $A_q(m,n)$. Then our lemma follows from  Lemma \ref{quadratic_relations}.

\end{proof}

By this lemma, the algebra  $A(V_m,V_n)$ is a deformation of the
ring of functions on the space of $m\times n$ matrices over $\kk$.  Indeed by the above lemma when $q=1$ we
have
$
A(V_m,V_n)\cong\sO(\Hom(\kk^n,\kk^m)),
$
the algebra of functions on $\Hom(\kk^n,\kk^m)$.

Since from now on we will only be working with the standard Yang-Baxter spaces we will drop the $V$ from the notation and write:
\begin{align*}
A_q(m,n) &= A(V_m,V_n) \\
A_q(m,n)_d &= A(V_m,V_n)_d \\
S_q(m,n;d)&= S(V_m,V_n;d)
\end{align*}
We refer to $A_q(m,n)$ as the algebra of  \textbf{quantum $m\times n$ matrices}. Note that when $m=n$ $A_q(n,n)$ is a bialgebra with counit $\epsilon:A_q(n,n)\to
\kk$ given by $\epsilon(x_{ij})=\delta_{ij}$.  In fact, $A_q(n,n)$ is the well-known
algebra of quantum $n\times n$ matrices (cf. \cite[\S 4]{T}). 

We now record a monomial basis of $A_q(m,n)$.  This is easiest to formulate using the following ordering.  Consider the set $\{x_{ji}:i,j=1,2,...\}$ of infinitely many variables with a total order so that 
$$x_{11}<x_{21}<x_{22}<x_{31}<x_{22}<x_{13}<x_{41}<\cdots$$
This induces a total order on $\{x_{ji}:1\leq j\leq m, 1\leq i \leq n \}$.  Now given a monomial $m=\prod_{ij}x_{ij}^{a_{ij}}\in A_q(m,n)$ let $\vec{m}$ be the reordered monomial so that the variables appear from smallest to biggest.  For instance, if $m=x_{21}^2x_{11}x_{31}^2$ then $\vec{m}=x_{11}x_{21}^2x_{31}^2$.

\begin{lemma}
\label{Basis_quantum_matrix}
The set of ordered monomials $\{\vec{m}:m=\prod_{ij}x_{ij}^{a_{ij}}, a_{ij}\geq0\}$ is a basis of $A_q(m,n)$ .
\end{lemma}

\begin{proof}
We apply the Bergman Diamond Lemma \cite{B}.  The set-up is as follows.  Let $X=\{x_{ji}:1\leq j\leq m, 1\leq i \leq n \}$ and let $\left< X\right>$ be the free monoid generated by $X$.   Endow $X$ with the reverse total order as the one above; $\left< X\right>$ is then endowed with the induced lexicographic total order 

  Let $\mathcal{S}$ be the set of relations from Lemma \ref{Qpresent}.  Every relation in $\mathcal{S}$ is of the form $m-f$ where $m\in \left< X\right>$, $f\in \kk\left< X\right>$ and $m$ is strictly bigger than every monomial in $f$ ($m$ is simply the leftmost monomial in each one of the relations above).  In other words, this order is ``compatible with reductions'' in the sense of \cite{B}. (Recall for a relation $m-f$ the corresponding reduction is an endomorphism of $\kk\left< X\right>$ that maps $AmB \mapsto AfB$ and every other element of $\left< X\right>$ to itself.)
  
  Note that the irreducible monomials, i.e. those unchanged by all reductions, are precisely the ordered monomials in the statement of the lemma.  Therefore by the Diamond Lemma, to conclude that these form a basis of $A_q(m,n)$ we need to show that one can resolve all minimal ambiguities.  This means that any sequence of reductions that one can apply to a degree three monomial $x_{ik}x_{j\ell}x_{rs}$ results in the same irreducible monomial.  This is a straightforward case-by-case analysis.  
\end{proof}

As a consequence of Lemma \ref{Basis_quantum_matrix},  as $\kk$-modules $A_q(m,n)$ and $A_q(m,n)_d$ are free over $\kk$.  

Consider $\Delta_{\ell,m,n}=\Delta_{V_\ell, V_m,V_n}:A_q(\ell,n)\to A_q(\ell,m)\otimes A_q(m,n)$ defined as in (\ref{co-mult}).  On generators $\Delta_{\ell,m,n}$ is given by
$$x_{ij}\mapsto \sum_{k=1}^{m} x_{ik}\otimes x_{k j} .$$
Usually $\ell,m,n$ are clear from context and we omit
them from the notation.

Recall from (\ref{Schur}) that 
$$
S_q(m,n; d)=(A_q(n,m)_d)^*.$$   Note that $S_q(n,n;d)$ is an algebra with the multiplication from (\ref{Schur_bilinear}), and it is the well-known  q-Schur algebra   (cf. \cite[\S 11]{T}), which is usually denoted $S_q(n,d)$.  Thus we can regard $S_q(m,n; d)$ as a kind of ``rectangular $q$-Schur algebra'' generalizing the $m=n$ case.  These will serve as the morphism spaces in the quantum divided power category which we define below.      

%We first record a PBW basis for quantum $m\times n$ matrices.
%\begin{lemma}
%\label{Mon_Basis}
%Ordered monomials in $x_{ij}$ (where $1\leq i \leq m$ and $1\leq j \leq n$)
%form a basis of $A_q(m,n)$.    
%\end{lemma}
%
%\begin{proof}
%Without loss of generality suppose $n\geq m$.  Then $A_q(m,n)$ is a subalgebra
%of $A_q(n)$.  It is well known that $A_q(n)$ has a basis consisting of ordered
%monomials in $x_{ij}$ (where $1\leq i,j \leq n$) (Theorem 3.5.1 in \cite{PW}).
% Therefore the monomials where $i\leq m$ are linearly independent as elements
%of $A_q(m,n)$.  Since these monomials clearly span $A_q(m,n)$, this proves
%the lemma.  
%\end{proof}

Let $\epsilon_n: A_q(n,n)_d\to \kk$ be the restriction of $\epsilon: A_q(n,n)\to \kk$.  The following lemma is well-known.
\begin{lemma}
\label{counit_lemma}
$\epsilon_n$ is the unit of the $q$-Schur algebra $S_q(n,n;d)$.
\end{lemma}
\begin{proof}
It is enough to check that $\epsilon_n$ is a counit of $A_q(n,n)_d$, i.e.  
to check that the following diagrams commute: 
$$
\xymatrix{
A_q(n,n)_d\ar[r]^<<<<<{\Delta}\ar[dr]^{\simeq} & A_q(n,n)_d\otimes A_q(n,n)_d\ar[d]^{1\otimes \epsilon_n}\\
& \kk\otimes A_q(n,n)_d
},\quad  \xymatrix{A_q(n,n)_d\ar[r]^<<<<<{\Delta}\ar[dr]^{\simeq } & A_q(n,n)_d\otimes A_q(n,n)_d\ar[d]^{\epsilon_n\otimes 1}\\
& A_q(n,n)_d\otimes \kk}
$$
But these diagrams are just the degree $d$ part of the left co-unit and right co-unit diagrams for $A_q(n,n)$, where $\epsilon$ is the co-unit of $A_q(n,n)$, and hence are known to be commutative \cite[Section 3.6]{PW}.   
\end{proof}

\section{Main definitions}
\label{sec-def}

\subsection{Classical polynomial functors}
\label{sec:classical}
Let $\Vect$ be the category of finite projective $\kk$-modules. 
To motivate our definition of quantum polynomial functors we first recall the classical category of strict polynomial functors.
For any $V\in\Vect$
the symmetric group $\fS_d$ acts on the tensor product $V^{\otimes d}$ by permuting factors.

For $V\in\V$ the
$d$-th divided power of $V$ is defined as the invariants $\Gamma^d(V)=(\otimes^dV)^{\fS_d}$.
 %Note that there are natural isomorphisms
%\Gamma^d(V)^*\cong S^d(V^*).$$
 Let $\Gamma^d\Vect$ denote the category consisting of objects $V\in\Vect$
and morphisms $$\Hom_{\Gamma^d\Vect}(V,W)=\Gamma^d(\Hom(V,W)).$$  The diagonal
inclusion $\fS_d \subset \fS_d\times\fS_d$ induces a morphism $$\Gamma^d(U)\otimes\Gamma^d(V)\to\Gamma^d(U\otimes
V).$$ Composition 
in $\Gamma^d\Vect$ is then defined as
$$
\xymatrix{
\Gamma^d(\Hom(V,U))\otimes \Gamma^d(\Hom(W,V)) \ar[r] & \Gamma^d(\Hom(V,U)\otimes\Hom(W,V))
\ar[d] \\
&  \Gamma^d(\Hom(W,U)).
}
$$
Let $\PP^d$ be the category consisting of $\kk$-linear functors $\Gamma^d\Vect
\to \Vect$.  Morphisms $\PP^d$ are natural transformations of functors. 
$\PP^d$ is the category of \textbf{polynomial functors of homogenous degree $d$} 

We remark that this is not the  definition of $\PP^d$ which originally appears 
in Friedlander and Suslin's work \cite{FS} on the finite generation of the cohomology of finite group schemes.  In their presentation polynomial
functors have both source and target the category $\Vect$, and it is required
that maps between $\Hom$-spaces are polynomial.  In the presentation we use, the polynomial condition is encoded in the category $\Gamma^d\Vect$.  For details see \cite{Kr, Ku} and references therein.

\subsection{Definition of quantum polynomial functors}

Note that in the above setup, $\Gamma^d(\Hom(V,W)) \cong \Hom_{\fS_d}(V^{\otimes d},W^{\otimes d}).$  This observation motivates our definition of quantum polynomial functors.

For any $d\geq 0$,  we define \textbf{quantum divided power category} $\Gamma_q^d\V$: 
 it consists of objects $0,1,2,...$ and
the morphisms are defined as 
\begin{equation}
\label{eq:mors}
\Hom_{\Gamma^d_q\Vect} (m,n):=\Hom_{\cB_d}(V_m^{\otimes d}, V_n^{\otimes d}).
\end{equation}
We should think of $\Gamma_q^d\V$ as the category of standard Yang-Baxter spaces $(V_n, R_n)$, and morphisms are given by $d$-th degree part of quantum Hom-space algebras. 

A \textbf{quantum polynomial functor of degree $d$} is defined to be a $\kk$-linear functor
$$F: \Gamma^d_q\V\to \V .$$
We denote by $\PP^d_q$ the category of quantum polynomial functors of degree $d$. Morphisms are natural transformations of functors.  

The category  $\PP_q^d$ is  an exact category in the sense of Quillen.  Foror the basics on exact categories see \cite{Bu}. Let $\PP_q$ be  the category of  quantum polynomial functors of all possible degrees, 
$$\PP_q:=\bigoplus_{d} \PP_q^d. $$
Given $F\in\PP_q$ we denote the map on hom-spaces by $F_{m,n}:\Hom_{\cB_d}(V_m^{\otimes d}, V_n^{\otimes d}) \to \Hom(F(m),F(n))$.

\begin{remark}
When $q=1$ our construction recovers the classical category $\PP^d$.  Indeed the natural functor $\Gamma_1^d\V \to \Gamma^d\V$ defined by $n\mapsto \kk^n$ is an equivalence of categories, and  induces an equivalence $\PP_1^d\cong\PP^d$.
\end{remark}

\begin{remark}
\label{rem:Hecke}
In the definition of the morphisms (\ref{eq:mors}) in $\Gamma^d_q\V$ we can replace $\cB_d$ by $\cH_d$ since the action of $\cB_d$ on tensor powers of the standard Yang-Baxter space factors through $\cH_d$.  
\end{remark}

$\PP_q$ has a monoidal structure.
For any $F\in \PP_q^d$ and $G\in \PP_q^e$  define the tensor product $F\otimes G\in \PP_q^{d+e}$ as follows:
 for any $n$, $(F\otimes G)(n):=F(n)\otimes G(n)$ and for any $m,n$, the map on morphisms is given by the composition
\begin{equation*}
\xymatrix{
\Hom_{\cB_{d+e}}(V_m^{\otimes d+e}, V_n^{\otimes d+e})\ar[r] & \Hom_{\cB_d\otimes \cB_e}(V_m^{\otimes d}\otimes V_m^{\otimes e}, V_n^{\otimes d}\otimes V_n^{\otimes e})  \ar[d] \\
& \Hom_{\cB_d}(V_m^{\otimes d},V_n^{\otimes d})\otimes \Hom_{\cB _e}(V_m^{\otimes e},V_n^{\otimes e}) \ar[d]^{F_{m,n}\otimes G_{m,n}} \\
& \Hom(F(m),F(n))\otimes\Hom(G(m),G(n)) \ar[d]\\
&  \Hom(F(m)\otimes G(m),F(n)\otimes G(n))
},
\end{equation*}
where the second morphism is in fact an isomorphism, which follows from the following general lemma.
\begin{lemma}
Let $A$ and $B$ be  $\kk$-algebras. Given  $A$-modules  $V_1,V_2$ and  $B$-modules $W_1,W_2$  such that $V_1,V_2,W_1,W_2$ are free over $\kk$ of finite rank, then the natural inclusion 
$$\alpha: \Hom_A(V_1,V_2)\otimes \Hom_B(W_1,W_2)\to  \Hom_{A\otimes B}(V_1\otimes W_1,V_2\otimes W_2))$$ is an isomorphism.
\end{lemma}
\begin{proof}
First of all we can identify $\Hom(V_1,V_2)\otimes \Hom(W_1,W_2)\simeq \Hom(V_1\otimes W_1, V_2\otimes W_2)$.  Hence the injectivity of $\alpha$ is clear.  Given any $f\in  \Hom_{A\otimes B}(V_1\otimes W_1,V_2\otimes W_2))$, we can write $f$ as $\sum_{i}  e_i\otimes \psi_i$, where $\{e_i\}$ is a basis of $\Hom(V_1,V_2)$,  and for each $i$, $\psi_i\in \Hom (W_1,W_2)$.  By assumption $f$ intertwines with the action of $1\otimes B$.  Since $\{e_i\}$ is a basis,  it follows that for every $i$,  $\psi_i$ intertwines with the action of $B$, i.e.  $f\in \Hom(V_1,V_2)\otimes \Hom_B(W_1,W_2)$.  Now we can write $f=\sum_{j}   \phi_j\otimes  a_j $, where $\{a_j\}$ is a basis of $\Hom_B(W_1,W_2)$. Note that $f$ also intertwines with $A\otimes 1$. It folllows that for any $j$, $\phi_j$ intertwines with the action of  $A$.    It shows the surjectivity of the inclusion $\alpha$.
\end{proof}

A duality is defined on $\PP_q$ as follows.  We first identify $V_m \cong V_m^*$ via the standard basis $e_i$, i.e. if $e_1^*,...,e_m^*$ denotes the dual basis of $V_m^*$ then $V_m\to V_m^*$ is given by $e_i \mapsto e_i^*$.  This induces an identification
$$
\sigma: \Hom_{\cB_d}(V_m^{\tn d},V_n^{\tn d}) \to \Hom_{\cB_d}(V_n^{\tn d},V_m^{\tn d}).
 $$ 
 For $F\in\PP_q^d$ we define $F^\sharp\in\PP_q^d$ by:
\begin{enumerate}
\item[(i)] $F^\sharp(n):=F(n)^*$,
\item[(ii)] $F^\sharp_{m,n}:\Hom_{\cB_d}(V_m^{\tn d},V_n^{\tn d}) \to \Hom(F^\sharp(m),F^\sharp(n))$ is given by the composition
$$
\xymatrix{
\Hom_{\cB_d}(V_m^{\tn d},V_n^{\tn d}) \ar[r]^{\sigma} & \Hom_{\cB_d}(V_n^{\tn d},V_m^{\tn d}) \ar[r]^{F_{n,m}} & \Hom(F(n),F(m)) \ar[d]^\cong \\
&& \Hom(F(m)^*,F(n)^*)
}
$$
\end{enumerate}
Given a morphism $f:F\to G$ in $\PP_q$, we define $f^\sharp:G^\sharp \to F^\sharp$ by $f^\sharp(n)=f(n)^*$.  It is straightforward to check that $f^\sharp$ is a morphism of polynomial functors. Note that the funcotr $*$ is a contravariant duality functor on $\Vect$.  Therefore $\sharp$ defines a contravariant duality $\sharp:\PP_q\to\PP_q$.

\begin{lemma}
\label{Duality_Tensor}
Given any two quantum polynomial functors $F$ and $G$ of homogeneous degree, then we have a canonical isomorphism 
$$(F\otimes G)^\sharp \simeq  F^\sharp \otimes G^\sharp. $$
\end{lemma}
\begin{proof}
The lemma is routine to check. It follows from the constructions of tensor product $\otimes$ and the contravariant functor  $\sharp$.
\end{proof}

\subsection{Examples}
\label{sec:exs}
The \textbf{identity functor} $I\in\PP_q^1$ is given by $I(n)=V_n$.  On morphisms it is the identity map.
We denote by $\bigotimes^d$ the $d$-th \textbf{tensor product functor}.  It is given by 
$n\mapsto  V_n^{\otimes d}$ and on morphisms by the natural inclusion 
$$\Hom_{\cB_d}(V_n^{\otimes d}, V_m^{\otimes d})\to \Hom(V_n^{\otimes d}, V_m^{\otimes d}). $$Notice that $\bigotimes^d=I^{\tn d}$.
It is also easy to see that the right action of $\cB_d$ on $V_n^{\otimes d}$ gives rise to endomorphisms of $\bigotimes^d$ as quantum polynomial functors, i.e. for any $w\in \fS_d$, $T_w: \bigotimes^d\to \bigotimes^d$ is a morphism. 

An important role will  be played by the the  functors  
$$\Gamma_q^{d,m}: n\mapsto  \Hom_{\cB_d}(V_m^{\otimes d},V_n^{\otimes d}).$$  Note that $\Gamma_q^{d,m}(n)= \Hom_{\cB_d}(V_m^{\otimes d}, V_n^{\otimes d} ).$  By   Proposition \ref{Natural_Iso} and Lemma \ref{Basis_quantum_matrix},  $\Gamma_q^{d,m}$ is a well-defined object in $\PP_q$. 
In particular when $m=1$,  it gives the  \textbf{$d$-th q-divided power} $\Gamma^d_q$.   

Let $\chi_+$ be the character of $\cB_d$ given by $\chi_+(T_i)=q$, and let $\chi_-$ be the character given by $\chi_-(T_i)=-q^{-1}$.  
We define  the \textbf{$d$-th q-symmetric power}   $S_q^d$ by
$$n\mapsto V_n^{\otimes d}\otimes_{\cB_d}  \chi_+, $$
 and the \textbf{$d$-th $q$-exterior power} $\bw_q^d$ by 
 $$n\mapsto   \bw_q^d:=V_n^{\otimes d}\otimes_{\cB_d}  \chi_-.$$ 
 For any $n$,  $ S_q^d(n) $ and $\bw_q^d(n)$ are free $\kk$-modules of finite rank, hence $S_q^d, \bw_q^d$ are examples in $\PP_q$.   

The quantum polynomial functors  $\Gamma^d_q, S^d_q$ and  $\bw_q^d$ are quantum analogues of divided power, symmetric power, exterior power functors. Moreover $(S^d_q)^\sharp\simeq \Gamma_q^d$.  Indeed when $q=1$ we recover the classical divided power, symmetric power and exterior power strict polynomial functors.  For instance, $\Gamma^d_1(n)=(V_n^{\otimes d})^{\fS_d}$, which is precisely the $d$-th divided power of $V_n$.  Moreover, $S^d_q(n)$ and  $\bw_q^d(n)$ recover  the constructions of quantum symmetric and exterior powers due to Berenstein and Zwicknagl \cite{BZ}.   

We remark also that since we are using the standard Yang-Baxter spaces the action of $\cB_d$ on  $V_n^{\otimes d}$ factors through the Hecke algebra $\cH_d$.  Therefore we could have replaced the occurrences of the braid group above by the Hecke algebra.  Now note that the characters  $\chi_{\pm}$ are the only two rank one modules of the Hecke algebra $\cH_d$, and they are  the quantum analogues of trivial and sign representation of symmetric group $\fS_d$.  The characters $\chi_\pm$ are used here to define the quantum symmetric and exterior powers in the same way that the trivial and sign representations are used to define the classical symmetric and exterior powers.

\subsection{An equivalent characterization of quantum polynomial functors}
Given a quantum polynomial functor $F$ of degree $d$ we get a finite projective $\kk$-module $F(n)$ for any $n\geq0$ and for any $m,n$, by Proposition \ref{Natural_Iso}, we  get a map: 
$$F_{m,n}: S_q(m,n;d)\to \Hom(F(m),F(n)).$$
This gives rise to maps
$$F'_{m,n}: S_q(m,n;d)\otimes F(m)\to F(n),$$
and also
$$F''_{m,n}:F(m)\to F(n)\otimes A_q(n,m)_d . $$
The following proposition gives an equivalent characterization of quantum polynomial functors in terms of the quantum matrix algebra.
\begin{proposition}
\label{prop-char}
A quantum polynomial functor $F$ of degree $d$ is equivalent to the following data: 
 \begin{enumerate}
\item for each positive integer a finite projective $\kk$-module $F(n)\in
\Vect$;
\item given any two nonnegative integers $m,n$
a $\kk$-linear map 
$$F''_{m,n}:F(m)\to  F(n)\otimes A_q(n,m)_d
$$
\end{enumerate}
such that, for any $\ell,m,n$, the following diagrams
commute
\begin{equation}
\label{functor_composition}
\xymatrix{
F(\ell)\ar[r]^{F''_{\ell,n}} \ar[d]^{F''_{\ell,m}}& 
 F(n)\otimes A_q(n,\ell)_d   \ar[d]^{ 1\otimes\Delta_{n,m,\ell} } \\
   F(m)\otimes A_q(m,\ell)_d  \ar[r]^<<<<<{1\otimes F''_{m,n}}
& F(n)\otimes A_q(n,m)_d \otimes A_q(n,\ell)_d 
}
\end{equation}
and for any $n$,
\begin{equation}
\label{functor_unit}
\xymatrix{
F(n)\ar[d]^{id}\ar[r]^<<<<<{F''_{n,n}} & 
  F(n)\otimes A_q(n,n)_d  \ar[dl]^{1\otimes \epsilon}\\
F(n)\otimes \kk 
}
\end{equation}
Here $\epsilon:A_q(n,n)_d \to \kk$ is the co-unit map (cf. Lemma \ref{counit_lemma}). 
\end{proposition}
\begin{proof}
By Proposition \ref{Natural_Iso},  for any $n,m$, $A_q(n,m)_d$ is dual to $\Hom_{\cB_d}(V_m^{\otimes d},  V_n^{\otimes d})$ as $\kk$-modules.  Given any element $\phi\in \Hom_{\cB_d}(V_m^{\otimes d},  V_n^{\otimes d})$, it is equivalent to give a $\kk$-linear functional $\tilde{\phi}: A_q(n,m)_d\to \kk $. 

Given a tuple of data $(F(n), F''_{m,n})$ which satisfies (\ref{functor_composition}) and (\ref{functor_unit}),   we can construct a quantum polynomial functor $F$, which assigns each $n\geq 0$ to $F(n)$, and on the level of morphisms, for any $\phi\in \Hom_{\cB_d}(V_m^{\otimes d},  V_n^{\otimes d})$, 
we set
 $$F(\phi):=(1_{F(n)}\otimes \tilde{\phi} )\circ F''_{m,n}.$$
  For any $\phi\in \Hom_{\cB_d}(V_m^{\otimes d},  V_n^{\otimes d}) $ and $\phi\in \Hom_{\cB_d}(V_n^{\otimes d},  V_\ell^{\otimes d})$,  we need to check that 
$$F(\psi\circ \phi)=F(\psi)\circ F(\phi) .$$
This follows from  (\ref{functor_composition}) and  Proposition  \ref{Iso_Com} by chasing diagrams.
Similarly $(\ref{functor_unit})$ implies that for the identity map $1\in \Hom_{\cB_d}(V_n^{\otimes d},  V_n^{\otimes d}) $, we have
$F(1)=1_{F(n)}$.    Therefore $F$ is a well-defined quantum polynomial functor.   

Conversely, given a quantum polynomial functor $F$,  we have explained in the beginning of this subsection how to get a tuple of data $(F(n), F''_{m,n})$, and (\ref{functor_composition}) and (\ref{functor_unit}) easily follow from the functor axioms.  
\end{proof}

%\subsection{Frobenius twist}

%In this subsection assume that $q$ is a primitive $\ell^{\text{th}}$ root of unity, where $\ell>1$ is an odd integer.  Define an algebra homomorphism
%$$
%\mathsf{Fr}:A_1(m,n)  \to A_q(m,n)  \text{ by } x_{ij} \mapsto x_{ij}^\ell.
%$$
%The same proof as in Lemma 7.2.2 in \cite{PW} gives that $\Delta_{m,r,n}:A_q(m,n) \to A_q(m,r) \otimes A_q(r,n) $ satisfies
%$$
%\Delta(x_{ij}^\ell)=\sum_{k=1}^r x_{ik}^\ell\otimes x_{kj}^\ell.
%$$
%Therefore the following diagrams commutes:
%$$
%\xymatrix{
%A_1(m,n)  \ar[r]^{\mathsf{Fr}} \ar[d]^{\Delta} & A_q(m,n)  \ar[d]^{\Delta} \\
%A_1(m,r) \otimes A_1(r,n)  \ar[r]^{\mathsf{Fr}\otimes\mathsf{Fr}} & A_q(m,r) \otimes A_q(r,n)
%}
%$$
%and
%$$
%\xymatrix{
%A_1(n,n)  \ar[r]^{\mathsf{Fr}} \ar[d]^\epsilon & A_q(n,n)  \ar[dl]^\epsilon \\
%\kk
%}
%$$
%Using this, we can define the \textbf{Frobenius twist} $(-)^{(1)}:\PP^d \to \PP_q^{\ell d}$.  Indeed given a  strict polynomial functor $F\in \PP^d$, as in the quantum case we get maps $$F''_{m,n}:F(m)\to F(n)\otimes A_1(n,m) .$$  Then set  $F^{(1)}(n)=F(n)$, and on morphisms define
%$$
%\xymatrix{
%(F^{(1)})_{m,n}'':F(m) \ar[r]^{F_{m,n}''}  & F(n)\otimes A_1(n,m)  \ar[r]^{1\otimes \mathsf{Fr}} & F(n)\otimes A_q(n,m) .
%}
%$$
%By Proposition \ref{prop-char} and the commutativity of the above two diagrams, $F^{(1)}$ is a quantum polynomial functor.  
%The definition of $(-)^{(1)}:\PP^d \to \PP_q^{\ell d}$ on natural transformations is straightforward.  

\section{Finite generation and representability}
\label{sec-mainthm}

\begin{definition}
The quantum polynomial functors $F\in\PP_q^d$ is \textbf{$m$-generated} if for any $n$ the map
$$
F_{m,n}': S_q(m,n;d) \otimes F(m) \to F(n)
$$
is surjective.  $F$ is \textbf{finitely generated} if it is $m$-generated for some $m$.
\end{definition}

Let $\underline{i}=\{i_1,...,i_r\}$ be a set of  positive integers.  Define a homomorphism
$$
\phi_{\underline{i}}:A_q(n,\ell) \to\kk
$$
by $x_{k\ell}\mapsto 1$ if $k=\ell$ and $k\in\underline{i}$, and otherwise $x_{k\ell}\mapsto0$. By Lemma \ref{Qpresent} $\phi_{\underline{i}}$ is a well-defined homomorphism of algebras .   By restriction we get a $\kk$-linear map
$$
\phi_{\underline{i}}^d:  A_q(n,m)_d \to\kk.
$$
In other words, $\phi_{\underline{i}}^d\in S_q(m,n;d)$.  For $F\in\PP_q^d$ we get a morphism 
 $$F_{m,n}(\phi_{\underline{i}}^d) \in \Hom(F(m),F(n)).$$

\begin{lemma}[Lemma 2.8, \cite{FS}]
\label{Lem-FS}
Let $V$ be a free $\kk$-module of finite rank.  We  fix elements $v_1,v_2,\cdots,v_n\in
V$. For any homogeneous polynomial $f\in S^d(V^*)$ of degree $d$, if $d< n$ then
$$f(v_1+v_2+\cdots+v_n)=\sum_{\und i\subset \{1,2,\cdots,n\}, |\und i|\leq d}(-1)^{n-|\und i|}f(\sum_{k\in \und i} v_k),$$ 
where $|\und i|$ is the cardinality of the set $\und i\subset \{1,2,\cdots,n\}$. 
\end{lemma}

\begin{lemma}
\label{lem-int}
If $m>d$ then $\phi_{\{1,...,m\}}^d\in S_q(m,m;d)$ is an integral linear combination of $\phi_{\underline{i}}^d$ where $|\underline{i}|\leq d$.
\end{lemma}
\begin{proof}

There is a  homomorphism of algebras $\delta: A_q(m,m)  \to \kk[t_{1},t_{2},\cdots,t_{m}]$ given by $x_{kk}\mapsto t_k$; $x_{k\ell}\mapsto 0$ if $k\not=\ell$. Note that for any $\und{i}$, 
 $\phi_{\und i}$ factors through the homomorphism $\tilde \phi_{\und i}: \kk[t_1,t_2,\cdots,t_m]\to \kk$, where 
$$\tilde \phi_{\und i}(t_k)=\begin{cases}
1 \quad \text{ if } k\in \und i \\
0 \quad \text{ otherwise} 
\end{cases}, $$
i.e. we have the following commutative diagram:
\begin{equation}
\xymatrix{ 
A_q(m,m)\ar[r]^{\delta}\ar[rd]^{\phi_{\und i}} & \kk[t_1,\cdots,t_m] \ar[d]^{\tilde \phi_{\und i}}\\
& \kk
}.
\end{equation}
Let $\tilde \phi_{\und i}^d$ be the restriction of $\tilde \phi_{\und i}$ to $\kk[t_1,t_2,\cdots,t_m]_d$, where $\kk[t_1,t_2,\cdots,t_m]_d$ is the space of homogenous polynomials in $t_1,t_2,\cdots,t_m$ of degree $d$.  Observe that for any polynomial $f\in \kk[t_1,t_2,\cdots,t_m]_d$,  we have $$\tilde \phi^d_{\und i}(f)= f(\sum_{i\in \und i}e_i),$$
 where  $e_{i}$ is the $i$-th basis in $\kk^m$.  Therefore the lemma follows from Lemma \ref{Lem-FS}.

\end{proof}

%Consider a triple of positive integers $\ell,m,n$ and suppose we have two sequences $\underline{i}=(i_1,...,i_r)$ and $\underline{j}=(j_1,...,j_s)$ such that $1\leq i_1<\cdots<i_r\leq\mathsf{min}(\ell,m)$ and $1\leq j_1<\cdots<j_s\leq\mathsf{min}(m,n)$.  
%Using the same formula as above we define $\phi_{\underline{i}}^d\in S_q(\ell,m,d)$ and $ \phi_{\underline{j}}^d\in S_q(m,n,d)$.  
%
%Define a new sequence $\underline{i}\cdot\underline{j}=(k_1,...,k_t)$ where $k_1$ is the first number in $\underline{i}$ that appears also in $\underline{j}$, $k_2$ is the second number in $\underline{i}$ that appears also in $\underline{j}$, and so forth.  Notice that $t\leq\mathsf{min}(\ell,n)$.  

\begin{lemma}
\label{lem-comp}
Let $\underline{i},\underline{j}$ be sets of positive integers and consider $\phi_{\underline{i}}^d\in S_q(\ell,m;d)$ and $\phi_{\underline{j}}^d \in S_q(m,n;d)$.  Furthermore consider $\phi_{\underline{i}\cap\underline{j}}^d\in S_q(\ell,n;d)$.
Then we have
$$
\phi_{\underline{j}}^d\circ \phi_{\underline{i}}^d=\phi_{\underline{i}\cap\underline{j}}^d.
$$
Therefore $F_{m,n}(\phi_{\underline{j}}^d)\circ F_{\ell,m}(\phi_{\underline{i}}^d) =F_{m,n}(\phi_{\underline{i}\cap\underline{j} }^d)    $.
\end{lemma}
\begin{proof}
It suffices to show that $
(\phi_{\underline{j}}\otimes \phi_{\underline{i}})\circ \Delta_{n,m,\ell}=\phi_{\underline{i}\cap\underline{j}}
$, and for this it suffices to show that both sides of the equation agree on $x_{ab}\in A_q(n,\ell)$:
\begin{align}
(\phi_{\underline{j}}\otimes \phi_{\underline{i}})(\Delta_{n,m,\ell}(x_{ab})) &= (\phi_{\underline{j}}\otimes \phi_{\underline{i}})(\sum_{p=1}^m x_{ap}\otimes x_{pb}) \\
&= \sum_{p=1}^m \phi_{\underline{j}}(x_{ap}) \phi_{\underline{i}}(x_{pb}) 
\end{align}
Since $\phi_{\underline{j}}(x_{ap}) \phi_{\underline{i}}(x_{pb})=1$ if and only if $a=b=p$ and $a\in \underline{i}\cap\underline{j}$ we have that 
$$
\sum_{p=1}^m \phi_{\underline{j}}(x_{ap}) \phi_{\underline{i}}(x_{pb}) =\phi_{\underline{i}\cap\underline{j}}(x_{ab}).
$$
The second statement of the lemma follows immediately. 
\end{proof}

\begin{proposition}
\label{prop-gen}
$F\in\PP_q^d$ is $m$-generated for any $m\geq d$.
\end{proposition}
\begin{proof}
We need to show that $F_{m,n}':S_q(m,n;d)\otimes F(m) \to F(n)$ given by $\phi\otimes v \mapsto F_{m,n}(\phi)(v)$ is surjective for any $n$.  

Suppose $m\geq n$ and choose $\underline{i}=\{1,...,n\}$.  By Lemma \ref{lem-comp}, $F_{n,n}(\phi_{\underline{i}}^d) =F_{m,n}(\phi_{\underline{i}}^d) \circ F_{m,n}(\phi_{\underline{i}}^d) $.  Now note that $\phi_{\underline{i}}^d\in S_q(n,n;d)$  is the unit element by Lemma \ref{counit_lemma}, and hence $F_{n,n}(\phi_{\underline{i}}^d) =1_{F(n)}$.  Therefore $F_{m,n}(\phi_{\underline{i}}^d) $ is surjective which  implies that $F_{m,n}'$ is as well.

Now suppose $m<n$.  By Lemma \ref{lem-int} the identity operator $1_{F(n)}$ is an integral linear combination of $F_{m,m}(\phi_{\underline{i}}^d) $, where $|\underline{i}|\leq d$.  Therefore we have, by Lemma \ref{lem-comp},
\begin{align*}
1_{F(n)} &= \sum_{|\underline{i}|\leq d}a_{\underline{i}}F_{n,n}(\phi_{\underline{i}}^d)  \\
&= \sum_{|\underline{i}|\leq d}a_{\underline{i}}F_{m,n}(\phi_{\underline{i}}^d) \circ F_{n,m}(\phi_{\underline{i}}^d) ,
\end{align*}
where $a_{\underline{i}}\in\mathbb{Z}$ and only finitely many are nonzero.  Given $v\in F(n)$ let $v_{\underline{i}}=(F_{n,m}(\phi_{\underline{i}}^d) ) (v)$.  Then we have that $v=\sum_{|\underline{i}|\leq d}a_{\underline{i}}(F_{m,n}(\phi_{\underline{i}}^d) )(v_{\underline{i}})$, i.e.
$$
F_{m,n}'\left(\sum_{|\underline{i}|\leq d}a_{\underline{i}}\phi_{\underline{i}}^d\otimes v_{\underline{i}}\right)=v
$$
proving that $F_{m,n}'$ is surjective.  
\end{proof}

\begin{proposition}
\label{Yonedalemma}
For any $n\geq 0$, the divided power $\Gamma_q^{d,n}$ represents the evaluation functor $\PP_q^d \to \V$ given by $F\mapsto F(n)$, i.e. there exists a canonical isomorphism 
$$\Hom_{\PP_q^d}(\Gamma_q^{d,n}, F)\simeq F(n). $$
Hence $\Gamma_q^{d,n}$ is a projective object in $\PP_q^d$.
\end{proposition}

\begin{proof}
We first show that given $F\in\PP_q^d$ there are natural isomorphisms 
$$
\Hom_{\PP_q^d}(\Gamma^{d,n}_q,F)\cong F(n)
$$
for any $n$.  
Consider the map $\phi:F(n) \to \Hom_{\PP_q^d}(\Gamma^{d,n}_q,F)$ given by $w\mapsto\phi_w$, where $\phi_w:\Gamma^{d,n}_q \to F$ is the natural transformation
$$
\phi_w(-)=\mathsf{ev}_w\circ F_{n,-}.
$$
In other words, $\phi_w(m):\Gamma^{d,n}_q(m) \to F(m)$ is the map 
$$
x\in \Hom_{\cB_d}(V_n^{\otimes d},V_m^{\otimes d}) \mapsto F_{n,m}(x)(w) \in F(m).
$$
Conversely, consider the map $\psi:\Hom_{\PP_q^d}(\Gamma^{d,n}_q,F) \to F(n)$ defined as follows:
\begin{align*}
f \in \Hom_{\PP_q^d}(\Gamma^{d,n}_q,F) & \leadsto f(n):\End_{\cB_d}(V_n^{\otimes d}) \to F(n) \\
& \leadsto f(n)(1_n) \in F(n)
\end{align*}
where $1_n\in \End_{\cB_d}(V_n^{\otimes d}) $ is the identity operator.  

Unpackaging these definitions we see that $\phi$ is inverse to $\psi$, proving that $\Gamma^{d,n}_q$ represents the evaluation functor.  It follows that $\Gamma^{d,n}_q$ is projective since the evaluation functor $ev_n: \PP^d_q\to \V$, $F\mapsto F(n)$ is exact.
\end{proof}

For an algebra $A$ we let $\Mod{A}$ denote the category of left $A$-modules that are finite projective over $\kk$.

\begin{theorem}
\label{Rep_thm}

If $n\geq d$ then $\Gamma^{d,n}_q$ is a projective generator of $\PP_q^d$.  Hence the evaluation functor $\PP_q^d \to \Mod{S_q(n,n;d)}$ is an equivalence of categories.

\end{theorem}

\begin{proof}
By Proposition \ref{Yonedalemma} we have that $\Gamma^{d,n}_q$ is projective.  To see that it's a generator when $n\geq d$ it suffices to show that $F_{n,-}':\Gamma^{d,n}_q \otimes F(n) \to F$ is surjective. This follows immediately from Proposition \ref{prop-gen}, which gives us that for every $m$ the map  $F_{n,m}'$ is surjective.  Hence the equivalence follows. 

\end{proof}

Let $\CoMod{A_q(n,n)_d}$ be the category of  right comodules over the coalgebra $A_q(n,n)_d$ that are finite projective over $\kk$.  It is clear that the category $\Mod{S_q(n,n;d)}$ is equivalent to $\CoMod{A_q(n,n)_d}$.  Therefore Theorem \ref{Rep_thm} immediately implies that 
the evaluation functor $\PP_q^d \to \CoMod{A_q(n,n)_d}$ is an equivalence if $n\geq d$. 

\begin{remark}
\label{rem:genthm}
Theorem \ref{Rep_thm} can be stated in slightly greater generality, where $\PP_q^d$ is replaced by the category of $\kk$-linear functors from $\Gamma_q^d$ to all projective $\kk$-modules, and $\Mod{S_q(n,n;d)}$ is replaced by $\mathsf{Mod}(S_q(n,n;d))$, the category of all $S_q(n,n;d)$-modules that are projective over $\kk$.  The same proofs carry over to this setting.  
\end{remark}

We now state a series of corollaries of Theorem \ref{Rep_thm}.  The first is well-known (cf. \cite[p.26]{BDK}) and it is an immediate consequence.
\begin{corollary}
Let $d\geq 0$ be an integer.  
For any two integers $m,n\geq d$ the $q$-Schur algebras $S_q(n,n;d)$ and $S_q(m,m;d)$ are Morita equivalent.
\end{corollary}

%\begin{proof}
%By Theorem \ref{Rep_thm}, $\PP_1 \cong \Mod{S_1(n,d)}$.  Note that $S_1(n,d)$ is the classical Schur algebra.  By Theorem 3.2 in \cite{FS} $\Mod{S_1(n,d)} \cong \PP^d$, completing the proof.
%\end{proof}

To state another corollary, we first note that the functor $\Gamma_q^{d,n}$ has a natural decomposition
\begin{equation}
\label{eq-decomp}
\Gamma_q^{d,n}\cong\bigoplus_{d_1+\cdots+d_n=d}\Gamma_q^{d_1}\otimes\cdots\otimes\Gamma_q^{d_n}.
\end{equation}
Indeed, by Frobenius reciprocity and Remark \ref{rem:Hecke} we have 
\begin{align*}
\Gamma_q^{d_1}(m)\otimes\cdots\otimes\Gamma_q^{d_n}(m) &\cong \Hom_{\cH_{d_1}\otimes\cdots\otimes\cH_{d_n}}(\chi_+\otimes\cdots\otimes\chi_+,V_m^{\otimes d}) \\
&\cong \Hom_{\cH_d}(\Ind_{\cH_{d_1}\otimes\cdots\otimes\cH_{d_n}}^{\cH_d}(\chi_+\otimes\cdots\otimes\chi_+),V_m^{\otimes d})
\end{align*}
and so (\ref{eq-decomp}) follows from the  isomorphism which is due to Dipper-James (\cite[Proposition 11.5]{T})
\begin{equation}
\label{eq-decomp2}
V_n^{\otimes d} \cong \bigoplus_{d_1+\cdots+d_n=d}\Ind_{\cH_{d_1}\otimes\cdots\otimes\cH_{d_n}}^{\cH_d}(\chi_+\otimes\cdots\otimes\chi_+).
\end{equation}
This isomorphism can be made explicit by mapping $e_1^{\otimes d_1}\cdots e_n^{\otimes d_n}$ to $1\in \Ind_{\cH_{d_1}\otimes\cdots\otimes\cH_{d_n}}^{\cH_d}(\chi_+\otimes\cdots\otimes\chi_+)$, and extending by $\cH_d$-linearity.

By Proposition \ref{Natural_Iso},  (\ref{eq-decomp2}) induces a partition of the unit of  $ S_q(n,n;d)$ into orthogonal idempotents: $1=\sum 1_{\vec{d}}$, where the sum ranges over all $\vec{d}=(d_1,...,d_n)$ such that $d_1+\cdots+d_n=d$.  For $M\in\Mod{S_q(n,n;d)}$ there is a corresponding decomposition into weight spaces
$$
M=\bigoplus M_{\vec{d}},
$$
where $M_{\vec{d}}=1_{\vec{d}}M$.

\begin{corollary}
\label{Weight_Rep}
Let $M\in\PP_q^d$, $n\geq0$ and $d_1,...,d_n\geq0$ such that $d_1+\cdots+d_n=d$.  Then   under the isomorphism $\Hom_{\PP_q}(\Gamma_q^{d,n},F)\cong F(n)$ we have
$$
\Hom_{\PP_q}(\Gamma_q^{d_1}\otimes\cdots\otimes\Gamma_q^{d_n},F)\cong F(n)_{(d_1,...,d_n)}.
$$
\end{corollary}
\begin{proof}
There is a canonical element $\iota_{(d_1,...,d_n)}\in \Gamma_q^{d_1}(n)\otimes\cdots\otimes\Gamma_q^{d_n}(n)$ corresponding to the inclusion 
$$
\Ind_{\cH_{d_1}\otimes\cdots\otimes\cH_{d_n}}^{\cH_d}(\chi_+\otimes\cdots\otimes\chi_+) \hookrightarrow V_n^{\otimes d}
$$
under (\ref{eq-decomp2}).  The map $\Hom_{\PP_q}(\Gamma_q^{d_1}\otimes\cdots\otimes\Gamma_q^{d_n},F)\to F(n)_{d_1,...,d_n}$ is given by
$f \mapsto f(n)(\iota_{(d_1,...,d_n)})$.  This map lands in the $(d_1,...,d_n)$ weight space since $f$ is a natural transformation.  More precisely, under our identifications we have the following commutative diagram:
$$
\xymatrix{
\Hom_{\cH_d}(\Ind_{\cH_{d_1}\otimes\cdots\otimes\cH_{d_n}}^{\cH_d}(\chi_+\otimes\cdots\otimes\chi_+),V_n^{\otimes d}) \ar[rr]^>>>>>>>>>>{f(n)} \ar[d]^{1_{(d_1,...,d_n)}} && F(n) \ar[d]^{1_{(d_1,...,d_n)}} \\
\Hom_{\cH_d}(\Ind_{\cH_{d_1}\otimes\cdots\otimes\cH_{d_n}}^{\cH_d}(\chi_+\otimes\cdots\otimes\chi_+),V_n^{\otimes d}) \ar[rr]^>>>>>>>>>>{f(n)} && F(n)
}
$$
which implies that $f(n)(\iota_{(d_1,...,d_n)})=1_{(d_1,...,d_n)}f(n)(\iota_{(d_1,...,d_n)})$.
Consider the  diagram
$$
\xymatrix{
\Hom_{\PP_q}(\Gamma_q^{d_1}\otimes\cdots\otimes\Gamma_q^{d_n},F) \ar[r] \ar[d] & F(n)_{d_1,...,d_n} \ar[d] 
\\
\Hom_{\PP_q}(\Gamma_q^{d,n},F)\ar[r] & F(n)
}
$$
This diagram clearly commutes.  Since both vertical maps are inclusions and the bottom map is an isomorphism by Theorem \ref{Rep_thm}, the top map is an isomorphism.
\end{proof}

The final corollary recovers a basic result relating the Hecke algebra and the $q$-Schur algebra.  Recall from Section \ref{sec:exs} that for any $w\in\fS_d$ we have a morphism of quantum polynomial functors $T_w:\bigotimes^d\to\bigotimes^d$.  Since the $T_i$ satisfy the Hecke relation this induces a map $\cH_d \to \Hom_{\PP_q}(\otimes^d,\otimes^d)$.  

\begin{corollary}
\label{Double_centrallizer}
The map $\cH_d\to  \Hom_{\PP_q}(\bigotimes^d,\bigotimes^d)$ is an isomorphism.  Hence for any $n\geq d$,  the map $\cH_d\to  \Hom_{S_q(n,n;d)}(V_n^{\otimes d}, V_n^{\otimes d})$ is an isomorphism of algebras. 
\end{corollary}
\begin{proof}
By Corollary \ref{Weight_Rep}, we have $\Hom_{\PP_d}(\bigotimes^d,\bigotimes^d)\simeq (V_d^{\otimes d})_{1,...,1}.$   The space $(V_d^{\otimes d})_{1,...,1}$ has of basis $e_{i_1}\otimes e_{i_2}\otimes \cdots \otimes e_{i_d}$, where $i_1,i_2,\cdots,i_d$ are all distinct. 
Under this isomorphism,  the map $\cH_d\to   (V_d^{\otimes d})_{1,...,1}$ is given by 
$T_w\mapsto   e_{w(1)}\otimes e_{w(2)}\otimes \cdots \otimes e_{w(d)}$, for any $w\in \fS_d$.  It is easy to see that this is a bijection. 

The second statement now follows from the first one using  Theorem \ref{Rep_thm}.
\end{proof}

Corollary \ref{Double_centrallizer} together with Proposition \ref{Natural_Iso}  recovers the double centralizer property between Hecke algebra and $q$-Schur algebra in the stable range when $n\geq d$.

\section{Braiding on $\PP_q$}
\label{sec-braiding}

In this section we will use Theorem \ref{Rep_thm} to define a braiding on the category of quantum polynomial functors, thus showing that $\PP_q$ is a braided monoidal category. 

Observe first that if $F\in\PP_q^d$ then, by Proposition \ref{prop-char}, the map $F''_{n,n}$ induces on $F(n)$ the structure of an $A_q(n,n)_d$-comodule:
$$
F_{n,n}'':F(n)\to F(n) \otimes A_q(n,n;d).
$$
We will use the Sweedler notation to denote this coaction: 
$$v\in F(n) \mapsto \sum v_0\otimes v_1\in F(n)\otimes A_q(n,n;d)$$  

For a coalgebra $C$ we let $\CoMod{C}$ be the category of  right $C$-comodules that are finite projective over $\kk$.
Now suppose we are given 
$$V\in \CoMod{A_q(n,n)_d}\quad  \text{ and }\quad W\in \CoMod{A_q(n,n)_e}.$$
 Then  $V\otimes W \in \CoMod{A_q(n,n)_{d+e}}$ and there is a well-known morphism induced from the R-matrix
$$
R_{V,W}:V\otimes W \to W\otimes V,
$$  
which is an isomorphism of $A_q(n,n)_{d+e}$-comodules.  We recall the construction of $R_{V,W}$ following Takeuchi \cite[\S 12]{T}.  

Define $\sigma: \Hom(V_n,V_n)\times \Hom(V_n,V_n)\ \to \kk$ by
\[\sigma(x_{ii},x_{jj}) = \left\{
  \begin{array}{lr}
    1  & \text{if }  i<j\\
    q & \text{if } i=j\\
   1 & \text{ if } i>j 
  \end{array}
\right.
\]
and in addition $\sigma(x_{ij},x_{ji})=q-q^{-1}$ if $i<j$ and $\sigma(x_{ij},x_{kl})=0$ otherwise.  

Recall that $\Hom(V_n,V_n)=A_q(n,n)_1\subset A_q(n,n)$, 
so we can extend $\sigma$ to a braiding on $A_q(n,n)$ \cite[Proposition 12.9]{T}. 
  This means that it is an invertible bilinear form on $A_q(n,n)$ such that for all $x,y,z\in A_q(n,n)$:
\begin{align*}
\sigma(xy,z) &=\sum\sigma(x,z_1)\sigma(y,z_2) \\
\sigma(x,yz) &=\sum\sigma(x_1,z)\sigma(x_2,y) \\
\sigma(x_1,y_1)x_2y_2 &=\sum y_1x_1\sigma(x_2,y_2)
\end{align*}
Here we again we use the Sweedler notation for the coproduct $\Delta:A_q(n,n)\to A_q(n,n)\otimes A_q(n,n)$ so  $\Delta(x)=\sum x_1\otimes x_2$.  
The R-matrix is given by $$R_{V,W}(v\otimes w)=\sum \sigma(v_1,w_1)w_0\otimes v_0.$$ 
Note that $R_{V_n,V_n}=R_n$, where $R_n$ is defined in Section \ref{sec-qmatrices}.

%Now recall that $\rho_{d+e,m}:\cH_{d+e}\to \End(V_m^{\otimes d+e})$ denotes the natural action. We also have that $R_{V_m^{\tn d},V_m^{\tn e}}\in\End(V_m^{\tn d+e})$.  
\begin{lemma}
\label{prop-im}
Let $d,e\geq0$.  Then there exists $\kappa \in \cH_{d+e}$ such that for all $m\geq1$ 
$$
R_{V_m^d,V_m^e}=\rho_{d+e,m}(\kappa).
$$
In particular $\kappa=T_{w_{d,e}}$ where $w_{d,e}\in\mathfrak{S}_{d+e}$ is given by
\[w(i) = \left\{
  \begin{array}{lr}
    i+e  & \text{if }  1 \leq i \leq d\\
    i-d & \text{if }  d < i 
  \end{array}
\right.
\]
\end{lemma}

\begin{proof}
For $$U\in \CoMod{A_q(m,m)_d}, V\in \CoMod{A_q(m,m)_e}, \text{ and } W\in \CoMod{A_q(m,m)_f}$$ the following two diagrams commute:
\begin{equation}
\label{eq-R2}
\xymatrix{
U\otimes V \otimes W \ar[rr]^{R_{U,V\otimes W}} \ar[dr]_{R_{U,V}\otimes 1_W} && V\otimes W\otimes U \\
& V\otimes U \otimes W \ar[ru]_{1_V\otimes R_{U,W}}
}
\end{equation}
and
\begin{equation}
\label{eq-R1}
\xymatrix{
U\otimes V \otimes W \ar[rr]^{R_{U\otimes V,W}} \ar[dr]_{1_U\otimes R_{V,W}} && W\otimes U\otimes V \\
& U\otimes W \otimes V \ar[ru]_{R_{U,W}\otimes 1_V}
}
\end{equation}
These are well-known properties of the R-matrix, and follow from the fact that $\sigma$ is a braiding.  

We will use these diagrams to prove the lemma by induction on $d+e$.  If $d+e=2$ then the statement is tautological.  If $d+e>2$ then suppose first $e\geq 2$.  By (\ref{eq-R2}) and the inductive hypothesis we have:
\begin{align*}
R_{V_m^{\tn d},V_m^{\tn e}} = R_{V_m^{\tn d},V_m^{\tn e-1}\otimes V_m} &=(1_{V_m^{\tn e-1}}\otimes R_{V_m^{\tn d},V_m})\circ (R_{V_m^{\tn d},V_m^{\tn e-1}}\otimes 1_{V_m}) \\
&= (1_{V_m^{\tn e-1}}\otimes \rho_{d+1,m}(T_{w_{d,1}})\circ (\rho_{d+e-1,m}(T_{w_{d,e-1}})\otimes 1_{V_m}) \\
&= \rho_{d+e,m}(T_{w_1})\circ \rho_{d+e,m}(T_{w_2}) \\
&=\rho_{d+e,m}(T_{w_1}T_{w_2})
\end{align*}  
where $w_1,w_2\in\mathfrak{S}_{d+e}$ are given by
\[w_1(i) = \left\{
  \begin{array}{lr}
    i & \text{if } 1\leq i\leq e-1 \\
    i+1  & \text{if }  e \leq i \leq e+d-1\\
    e & \text{if }  i=e+d
  \end{array}
\right.
\]
and
\[w_2(i) = \left\{
  \begin{array}{lr}
    e-1+i & \text{if } 1\leq i\leq d \\
    i-d  & \text{if }  d+1 \leq i \leq e+d-1\\
    e+d & \text{if }  i=e+d
  \end{array}
\right.
\]
Since $w_1w_2=w_{d,e}$ and $\ell(w_1)+\ell(w_2)=\ell(w_{d,e})$ (where $\ell$ is the usual length function), we have that $T_{w_1}T_{w_2}=T_{w_{d,e}}$ and the result follows.  

In the case that $e< 2$ then $d\geq 2$ and a similar induction applies, where one uses (\ref{eq-R1}) instead of (\ref{eq-R2}) .  
\end{proof}

Now suppose $F\in\PP_q^d$ and $G\in \PP_q^e$.  Define 
$$
R_{F,G}:F\otimes G \to G\otimes F
$$ 
by $R_{F,G}(m)=R_{F(m),G(m)}$.

\begin{theorem} 
\label{Braiding}
$R$ induces a braiding on the category $\PP_q$.  In other words, let $F\in\PP_q^d$
and $G\in \PP_q^e$.  Then $R_{F,G} \in \Hom_{\PP_q}(F\otimes G,G\otimes
F)$ and moreover $R_{F,G}$ is an isomorphism. 
\end{theorem}

\begin{proof}
We only need to show that $R_{F,G} \in \Hom_{\PP_q}(F\otimes G,G\otimes
F)$; the fact that $R_{F,G}$ is an isomorphism then follows immediately.

We first prove $R_{F,G} \in \Hom_{\PP_q}(F\otimes G,G\otimes F)$ in the
case where $F=\bigotimes^d$ and $G=\bigotimes^e$.  In that case we need to show that for
any $x\in \Hom_{\cH_{d+e}}(V_m^{\tn d+e},V_n^{\tn d+e})$ the diagram
\begin{equation}
\label{eq-R}
\xymatrix{
V_m^{\tn d}\otimes V_m^{\tn e} \ar[rr]^{\bigotimes^{d+e}(x)} \ar[d]_{R_{V_m^{\tn d},V_m^{\tn e}}} && V_n^{\tn d}
\otimes V_n^{\tn e} \ar[d]^{R_{V_n^{\tn d},V_n^{\tn e}}} \\
V_m^{\tn e}\otimes V_m^{\tn d} \ar[rr]^{\bigotimes^{d+e}(x)} && V_n^{\tn e}\otimes V_n^{\tn d} 
}
\end{equation}
commutes.
Cleary we have that  $\bigotimes^{d+e}(x)\in \Hom_{\cH_{d+e}}(V_m^{\tn d+e},V_n^{\tn d+e})$,
i.e. for all $\tau\in \cH_{d+e}$
$$
\xymatrix{
\bigotimes^{d+e}(x)\circ \rho_{d+e,m}(\tau)=\rho_{d+e,n}(\tau)\circ \bigotimes^{d+e}(x).
}
$$
In particular this is true for $\tau=\kappa$, which, by Lemma \ref{prop-im},
is precisely the commutativity of (\ref{eq-R}).

Now, by Theorem \ref{Rep_thm}, any $F\in \PP_q^d$ is a subquotient of some
copies of $\bigotimes^d$.  Therefore to prove the theorem in general it suffices
to prove it for $F=F'/F''$ and $G=G'/G''$ such that $F, G\in \PP_q$, where $F''\subset F'\subset \bigotimes^d$
and $G''\subset G'\subset \bigotimes^e$.  
In other words, we need to show that for $F$ and $G$ as in the previous sentence
and any $x\in \Hom_{\cH_{d+e}}(V_m^{\tn d+e},V_n^{\tn d+e})$ the diagram
\begin{equation*}
\xymatrix{
F(m)\otimes G(m) \ar[rr]^{F\otimes G(x)} \ar[d]_{R_{F(m),G(m)}} && F(n) \otimes
G(n) \ar[d]^{R_{F(n),G(n)}} \\
G(m)\otimes F(m) \ar[rr]^{G\otimes F(x)} && G(n)\otimes F(n) 
}
\end{equation*}
the diagram commutes.  This is a consequence of the commutativity of (\ref{eq-R})
and the fact that the R-matrix is compatible with restriction.  In other words, given $V\in \CoMod{A_q(m,m)_d}$ and $W\in \CoMod{A_q(m,m)_e}$ and sub-comodules $V'\subset V$ and
$W'\subset W$ then
$
R_{V',W'}=R_{V,W}|_{V'\otimes W'}.
$
\end{proof}

Let $\Omega(n,d)$ be the set of tuples $I=(i_1,i_2,\cdots,i_d)$, where $1\leq i_k\leq n$ for any $1\leq k\leq d$.  We call $I$ \textbf{increasing} if $i_1\leq i_2\leq \cdots \leq i_d$ and $I$ is \textbf{strictly increasing} if $i_1<i_2<\cdots < i_d$.  We denote by $e_I$ the element $e_{i_1}\otimes e_{i_2}\otimes \cdots \otimes e_{i_d}\in V_n^{\otimes d}$.
We now introduce a pairing $(,)$ on $V_n^{\otimes d}$, for any $I,J\in \Omega(n,d)$,
$$(e_I,e_J):=\delta_{IJ}, $$
where $\delta_{IJ}$ is the Kronecker symbol.

\begin{lemma}
\label{Contragredient_Braid}
Given  any $w\in \fS_d$, $I,J\in \Omega(n,d)$, we have 
$$(e_I\cdot T_w, e_J)=(e_I,e_J\cdot T_{w^{-1}}). $$
\end{lemma}
\begin{proof}
It can be reduced to the case $d=2$.  In this case, it suffices to check that for any $i,j,k,\ell$,  $$(R_n(e_i\otimes e_j), e_{k}\otimes e_\ell)=(e_i\otimes e_j, R_n(e_k\otimes e_\ell)).$$
This is a straightforward computation from the definition of the R-matrix $R_n$. 
\end{proof}

The following lemma follows from the definition of duality functor $\sharp $.
\begin{lemma}
\label{Iden_Tensor_power}
There exists a canonical isomorphism.
$(\bigotimes ^d)^\sharp\simeq \bigotimes ^d$.
\end{lemma}

By this lemma, we can identify $\bigotimes^d$ and $(\bigotimes^d)^{\sharp}$.

\begin{proposition}
\label{Braid_Duality}
Given any $w\in \fS_d$,  we have
$$ 
\xymatrix{
(T_w)^{\sharp}=T_{w^{-1}}:  \bigotimes^d\to \bigotimes^d .
}
$$
\end{proposition}
\begin{proof}
It follows from Lemma \ref{Contragredient_Braid} and Lemma \ref{Iden_Tensor_power}.
\end{proof}

The following proposition is about the compatibility between the duality functor $\sharp$ and the braiding $R$. 
\begin{proposition}
\label{R_matrix_duality}
Given any two quantum polynomial functors $F,G\in \PP_q$, we have 
$$(R_{F,G})^\sharp=  R_{G^\sharp, F^\sharp} .$$
\end{proposition}
\begin{proof}
It suffices to check the following diagram commutes,
\begin{equation}
\xymatrix{
(G\otimes F)^{\sharp}  \ar[d]^{\simeq} \ar[r]^{(R_{F,G})^\sharp} &  (F\otimes G)^\sharp  \ar[d]^{\simeq} \\
G^{\sharp}\otimes F^\sharp   \ar[r]^{R_{G^\sharp, F^\sharp}} &  F^\sharp \otimes G^\sharp },
\end{equation}
where the horizontal maps are the canonical isomorphisms in Lemma \ref{Duality_Tensor}.
By the functoriality of $R$, 
as the argument in Theorem \ref{Braiding} we can reduce to the case $F=\bigotimes^d$ and $G=\bigotimes^e$.   Under the identification $(\bigotimes ^n)^\sharp \simeq \bigotimes^n$ for any $n$, 
it is enough for us to check $(R_{\otimes^d, \otimes^e})^\sharp= R_{\otimes^e, \otimes^d}$.  By Theorem \ref{Braiding} and Proposition \ref{Braid_Duality}, we only need to show that $w_{d,e}^{-1}=w_{e,d}$, which is clearly true.  
\end{proof}

\section{Quantum Schur and Weyl Functors}
\label{sec: schurfunctors}
In this section we assume $q^2\neq -1$.  We define quantum Schur and Weyl functors.  As in the setting of classical strict polynomial functors, these families of functors play a fundamental role, and we use them here to construct the simple objects in $\PP_q$ (up to isomorphism).  In several key calculations in this section  we appeal to theorems in \cite{HH}.

\subsection{Quantum symmetric and exterior powers}
We call $I\in \Omega(n,d)$ \textbf{strict} if for any $1\leq k\neq \ell\leq d$, $i_k\neq i_\ell$.   Let $\Omega^{++}(n,d)$ be the set of strictly increasing tuples of integers in $\Omega(n,d)$.
We denote by $x_{IJ}$ the monomials 
$x_{i_1j_1}x_{i_2j_2}\cdots x_{i_d j_d}$ in $A_q(n,m)$ where $I=(i_1,i_2,\cdots,i_d)\in \Omega(n,d)$ and $J=(j_1,j_2,\cdots,j_d)\in \Omega(m,d)$. 

Recall that by Remark \ref{rem:Hecke} $\bigwedge_q^d(n)=V_n^{\tn d}\tn_{\cH_d}\chi_-$.  Note that $\bigwedge_q^d(n)$ is isomorphic to the $d^{\text{th}}$ graded component of
$$
\xymatrix{
\bigwedge_q^\bullet(n):=T(V_n)/I(R_n)
}
$$   
where $T(V_n)$ is the tensor algebra of $V_n$ and $I(R_n)$ is the two-sided ideal of $T(V_n)$, generated in degree two by $R_n(v\otimes w)+q^{-1}w\otimes v$, for 
$v,w\in V_n$.

As usual for exterior algebras, we use $\wedge$ to denote the product in the algebra $\bigwedge^\bullet_q(n)$.
For any $I\in \Omega(n,d)$ we denote by $\bar{e}_I$ the image of $e_I$ in $\bigwedge_q^d (n)$:
$$\bar{e}_I=e_{i_1}\wedge e_{i_2}\wedge \cdots \wedge e_{i_d}. $$
Moreover we have the following basic calculus of q-wedge products:
$$e_i\wedge e_j=\begin{cases} 

0   \qquad \text{ if } i=j \\
-q^{-1}e_j\wedge e_i \qquad \text{ if }  i>j
 \end{cases} $$

\begin{lemma}
\label{q_wedge_property}
 Let  $I=(i_1,i_2,\cdots,i_d)\in \Omega(n,d)$.
\begin{enumerate}
\label{q_wedge_cal}
\item 
If there exists $1\leq k\not= \ell\leq d$ such that $i_k=i_\ell$ then $e_{i_1}\wedge e_{i_2}\wedge \cdots \wedge e_{i_d}=0$. 
\item If $I$ is strictly increasing and  $\sigma\in\fS_d$, then 
$$e_{i_{\sigma(1)}}\wedge e_{i_{\sigma(2)}}\wedge \cdots \wedge e_{i_{\sigma(d)}} =(-q^{-1})^{\ell(\sigma)}e_{i_1}\wedge e_{i_2}\wedge \cdots \wedge e_{i_d}, $$ 
where $\ell(\sigma)$ is the length of $\sigma$.
\end{enumerate}
\end{lemma}
\begin{proof}
Both parts follow easily from the definition of the $q$ wedge products, cf. Equations (2.3),(2.4) in \cite{HH}.
%
%
%$(1)$ is clear.
%   We note that if $I$ is strictly increasing, then  $e_{i_{\sigma(1)}}\otimes e_{i_{\sigma(2)}}\otimes\cdots \otimes e_{i_{\sigma(d)}}=e_I\cdot T_{\sigma}$. Hence by the definition of $q$-wedge product $(2)$ also follows. 
\end{proof}

A consequence of above lemma is that   $\bigwedge_q^d(n)$ has a basis  $e_{i_1}\wedge\cdots\wedge  e_{i_d}$
for $1\leq i_1 <\cdots< i_d\leq n$. 
The $q$-antisymmetrization map $\alpha_d(n):\bigwedge_q^d(n)\to V_n^{\tn d}$
is given by
$$
e_{i_1}\wedge \cdots\wedge  e_{i_d} \mapsto \sum_{w\in \fS_d}(-q^{-1})^{\ell(w)}e_{i_{w(1)}}\tn\cdots\tn
e_{i_{w(d)}},$$
for $1\leq i_1 <\cdots< i_d\leq n$.

We define the following elements of $\cH_d$:
\begin{align*}
x_d&=\sum_{w\in \fS_d}q^{\ell(w)}T_w\\
y_d&=\sum_{w\in \fS_d}(-q^{-1})^{\ell(w)}T_w.
\end{align*}
In the current setting, it is convenient for us  to denote the right action of $\cH_d$ on $V_n^{\tn d}$ by a dot.  

\begin{lemma}
\label{q_anti_sym_computaion}
Given any tuple $I=(i_1,i_2,\cdots,i_d)\in \Omega(n,d)$ we have 
$$\alpha_d(n)(e_{i_1}\wedge\cdots\wedge  e_{i_d})=e_I\cdot y_d. $$
\end{lemma}
\begin{proof}

Suppose first that $I$ is strict.  Let $I_0$ be the strictly increasing tuple such that $I=I_0\cdot \sigma$ for a unique permutation $\sigma\in \fS_d$. The following computation proves the lemma in this case:
\begin{equation}
\begin{aligned}
\alpha_d(n)(\bar e_I) &=(-q^{-1})^{\ell(\sigma)}\alpha_q^d(\bar e_{I_0})\\
&=(-q^{-1})^{\ell(\sigma)} \sum_{w\in \fS_d}(-q^{-1})^{\ell(w)}e_{I_0\cdot w}\\
&= (-q^{-1})^{\ell(\sigma)}e_{I_0}\cdot y_d\\
&= e_{I_0}\cdot (T_{\sigma}\cdot y_d)\\
&=  e_I \cdot  y_d 
\end{aligned}
\end{equation}
where the first equality follows from Lemma \ref{q_wedge_cal}
(2), the third and the last equalities holds because $I_0$ is strictly increasing and the fourth equality follows from the following fact:
$$T_{\sigma}\cdot y_d=(-q^{-1})^{\ell(\sigma)}  y_d.$$

Now suppose that $I$ is not strict.  Then by  Lemma \ref{q_wedge_cal} (1) it is enough to
show 
\begin{equation}
\label{Zero_formula}
e_I\cdot y_d=0.
\end{equation}
Let $I=(i_1,i_2,\cdots,i_d)$. Assume that $k$ is the maximal number such that  $i_1,i_2,\cdots,i_k$ are all distinct but $i_{k+1}$ is equal to one of $i_1,i_2,\cdots,i_k$.
Let $\sigma$ be the (unique) element in $\fS_{k}\subset \fS_d$, such that $(i_{\sigma^{-1}(1)},i_{\sigma^{-1}(2)}.\cdots,i_{\sigma^{-1}(k)}) $ are strictly increasing. Then $e_I=e_{I\cdot \sigma^{-1}}T_\sigma$ and
 \begin{equation*}
\begin{aligned}
e_I\cdot y_d=e_{I\cdot \sigma^{-1}}(T_\sigma y_d) 
= (-q^{-1})^{\ell(\sigma)} e_{I\cdot \sigma^{-1}}\cdot y_d.
\end{aligned}
\end{equation*}
Hence to show the formula (\ref{Zero_formula}), we can always assume that $i_1<i_2<\cdots <i_k$ and  $i_{k+1}=i_a$, where $1\leq a\leq k$.   Take the element $S=T_{a+1}\cdots T_{k-1}T_{k}\in \cH_d$. Then $e_I=e_{I'}\cdot S$, where $I'=(i_1,i_2,\cdots
,i_a, i_{k+1},i_{a+1},i_{a+2},\cdots ,i_k, i_{k+2},i_{k+3},\cdots ,i_d)$ and then
\begin{equation*}
\begin{aligned}
e_I\cdot y_d =e_{I'}(Sy_d)
=(-q^{-1})^{k-a}e_{I'}\cdot y_d
\end{aligned}.
\end{equation*}
Note that $e_{I'}T_{a}=qe_{I'}$.  On the other hand
$$e_{I'}(T_a\cdot y_d) =(-q^{-1})(e_{I'}\cdot y_d).$$
By the assumption that $q^2\neq -1$, it forces 
$e_{I'}\cdot y_d=0$,
and hence 
$e_I\cdot y_d=0.$
\end{proof}

Recall also that we have the quantum symmetric power $$S_q^d(n)=V_n^{\tn d}\tn_{\cH_d} {\chi_+},$$ 
and the quantum divided power functor 
$$\Gamma_q^d(n)=\Hom_{\cH_d}(\chi_+, V_n^{\otimes d}). $$ 

Let $p_d$ be the projection map $p_d: \bigotimes^d\to \bigwedge^d_q$ and let $q_d$ be the  projection morphism $q_d: \bigotimes^d\to  S_q^d$. 
Let $i_d: \Gamma_q^d\to \bigotimes^d$ be the natural inclusion map.   It is clear that $p_d, q_d, i_d$ are morphisms of quantum polynomial functors.   

\begin{proposition}
\label{prop-antisym}
 The $q$-antisymmetrization $\alpha_d:\bigwedge_q^d \to \bigotimes^{ d}$ is a morphism of quantum polynomial functors. 
\end{proposition}
\begin{proof}
We work with the characterization of quantum polynomial functors given by Proposition \ref{prop-char}.
We need  to check that, for any $n,m$, the following diagram commutes: 
\begin{equation}
\label{Com_check}
\xymatrix{
\bigwedge^d_q(m)\ar[d]^{\alpha_d(n)} \ar[r] & \bigwedge^d_q(n)\otimes A_q(n,m)_d\ar[d]^{\alpha_d(n)\otimes 1} \\
V_m^{\otimes d} \ar[r]    &     V_n^{\otimes d}\otimes A_q(n,m)_d
}.
\end{equation}

The quantum polynomial functor $\bigotimes^d$ gives rise to the bottom map, which for any $I\in \Omega(m,d)$, is given by
$$e_{I}\mapsto \sum_{J\in \Omega(n,d)} e_J\otimes x_{JI}. $$
It  also induces the quantum polynomial functor structure on  $\bigwedge_q^d$, and so for any $m,n$ and for any $I\in \Omega(n,d)$ the top map is given by  
$$\bar{e}_I\mapsto  \sum_{J\in \Omega(n,d)} \bar{e}_J\otimes x_{JI}, $$
where  $\bar{e}_I\in \bigwedge_q^d(m)$  and  $\bar{e}_J\in \bigwedge_q^d(n)$.  

 We start with an element $\bar{e}_I\in \bigwedge^d_q(m)$, where $I$ is strictly increasing. 
 In the diagram (\ref{Com_check}), if we go up-horizontal and then downward, then by Lemma \ref{q_anti_sym_computaion}, $\bar{e}_I$ is mapped to 
\begin{equation}
\begin{aligned}
\sum_{J\in \Omega(n,d)}e_J\cdot y_d\otimes x_{JI} &=\sum_{J\in \Omega(n,d)}\sum_{w\in \fS_d} (-q^{-1})^{\ell(w)}e_J\cdot T_w\otimes x_{JI}\\
 &=\sum_{w\in
\fS_d} (-q^{-1})^{\ell(w)} (\sum_{J\in \Omega(n,d)}e_J\cdot T_w\otimes x_{JI})\\
 &=\sum_{w\in
\fS_d} (-q^{-1})^{\ell(w)} (\sum_{J\in \Omega(n,d)}e_J\otimes x_{J(I\cdot w)})
 \end{aligned},
 \end{equation}
 where the last equality holds since $T_w$ is an endomorphism of the quantum polynomial functor  $\bigotimes^d$, and also $e_I\cdot T_w=e_{I\cdot w}$.

If we go downward and then down-horizontal, $\bar{e}_I$ is exactly mapped to  
$$\sum_{w\in
\fS_d} (-q^{-1})^{\ell(w)} (\sum_{J\in \Omega(n,d)}e_J\otimes x_{J(I\cdot w)})$$
showing the commutativity of the diagram (\ref{Com_check}).
\end{proof}

%For $q$ generic define 
%$$
%\pi_q=\frac{{n\choose 2}}{[n]!}\sum_{w\in \fS_d}(-q^{-1})^{\ell(w)}T_w.
%$$
%This is an idempotent $\cH_d$ and
%$
%\pi_q(V_n^{\tn d}) \cong \Lambda_q^d(n).
%$
%Consequently the $\alpha_q^d$ is $S_q(n,d)$-invariant.
%\ocom{How do we know that $\alpha_q^d$ is $S_q(n,d)$-invariant for any $q\neq \sqrt{-1}$?  I conjecture that for any $q\neq \sqrt{-1}$
%we have $\pi_q'(V_n^{\tn d}) \cong \Lambda_q^d(n)$, where 
%$$
%\pi_q'=\sum_{w\in \fS_d}(-q^{-1})^{\ell(w)}T_w.
%$$
%This of course would be enough to show that  $\alpha_q^d$ is $S_q(n,d)$-invariant.
%}

\begin{proposition}
\label{symmetrization_projection}
There exist  canonical isomorphisms 
$$
\xymatrix{(\bw_q^d)^{\sharp}\simeq \bigwedge^{d}_q,\quad  (S^d_q)^\sharp\simeq \Gamma^d_q .}$$
Under these identifications, we have the following equalities: 
$$(p_d)^{\sharp}=\alpha_d,\quad 
(q_d)^\sharp=i_d.$$
\end{proposition}
\begin{proof}
We first consider $\bigwedge^d_q$.   Let $\{(\bar{e}_I)^*\}_{I\in \Omega(n,d)^{++} }$ be the dual basis of $\{\bar{e}_I\}_{I\in \Omega(n,d)^{++} } $ in $(\bigwedge^d_q(n))^*$, where $\bar {e}_I=e_{i_1}\wedge e_{i_2}\wedge \cdots \wedge e_{i_n}\in \bigwedge_q^d(n)$ for  $I=(i_1,i_2,\cdots,i_n)$.  It naturally gives a set of  elements in $(V_n^{\otimes d})^*$ via the inclusion map $(\bigwedge^d_q(n))^*\to (V_n^{\otimes d})^*$. By Lemma \ref{q_wedge_property},  the element $(\bar{e}_I)^*$ can be identified with 
$$e_I \cdot y_d=\sum_{w\in \fS_d} (-q^{-1})^{\ell(w)}e_{I\cdot w} .$$

It exactly coincides with the of image of $\bar{e}_I$ after the $q$-antisymmetrization map $\alpha_d(n)$.  It implies that $(\bigwedge_q^d)^\sharp \simeq \bigwedge^d_q$ under the correspondences $(\bar{e}_I)^*\mapsto  \bar{e}_I$, moreover  $\alpha_d=(p_d)^\sharp.  $

We now consider the projection map  $q_d: \bigotimes^d\to S^d_q$.  Note that the $q$-symmetric power $S_q^d(n)$ is the quotient 
$$\frac{V_n^{\otimes d}}{\sum_{i=1}^{d-1}\rm{Im}(T_i-q)}, $$
and the $q$-divided power  $\Gamma^d_q(n)$ is the  subspace of $V_n^{\otimes d}$,
$$\bigcap_{i=1}^{d-1} {\rm{Ker}}(T_i-q). $$  
By Lemma \ref{Contragredient_Braid}, the  operator  $T_i-q:  V_n^{\otimes d}\to V_n^{\otimes d}$ is self-adjoint, with respect to the bilinear form $(,)$.  Therefore 
$\frac{V_n^{\otimes d}}{\sum_{i=1}^{d-1}\rm{Im}(T_i-q)}$ is dual to $\bigcap_{i=1}^{d-1} {\rm{Ker}}(T_i-q)$.  In particular it also implies that 
$(S_q^d)^\sharp\simeq \Gamma_q^d $ and $(q_d)^\sharp=i_d .$
  \end{proof}
  
\subsection{Definition and properties of Quantum Schur and Weyl functors}
\label{SubsectionSchurWeyl}

Let $\la=(\la_1,...,\la_s)$ be a partition.  By convention our partitions have no zero parts, so $\la_1\geq\cdots\geq\la_s>0$.  
The size of $\la$ is $|\la|:=\la_1+\cdots+\la_s$ and the length of $\la$ is $\ell(\la):=s$.  
We  depict partitions using diagrams, e.g.
$(3,2)={\tiny\yng(3,2)}$.  Let $\la'$ denote the conjugate partition.  

%A tableau of shape $\la$ is row-standard if the entries strictly increase along rows from left to right.  We let $RST(\la,n)$ denote the set of row standard tableau of shape $\la$ with entries from $1,...,n$.   

%A tableau of shape $\la$ is semistandard if it is row-standard and the entries weakly increase down any column from top to bottom.  We let $ST(\la,n)$ denote the set of semistandard tableau of shape $\la$ with entries from $1,...,n$.  
The canonical tableau of shape $\la$ is the  tableau with entries $1,...,|\la |$ in sequence along the rows.  For example
$$
\young(123,45)
$$
is the canonical tableau of shape $(3,2)$.  Let $\sigma_\la\in\fS_d$ 
be given by the column reading word of the canonical tableau.  For instance, if $\la=(3,2)$ then $\sigma_\la=14253$ (in one-line notation).  Define the following quantum polynomial functors of degree $d$:
\begin{align*}
\xymatrix{\bigwedge^\lambda_q}&=\xymatrix{\bigwedge_q^{\lambda_1}\otimes\cdots\otimes\bigwedge_q^{\lambda_s}}\\
S^\lambda_q&=S_q^{\lambda_1}\otimes\cdots\otimes S_q^{\lambda_s}\\
\Gamma^\lambda_q&=\Gamma_q^{\lambda_1}\otimes\cdots\otimes \Gamma_q^{\lambda_s}
\end{align*}
and the following morphisms:
\begin{align*}
\alpha_\la&=\alpha_{\lambda_1}\otimes \alpha_{\lambda_2}\otimes \cdots \otimes \alpha_{\lambda_s}\\
i_\lambda&=i_{\lambda_1}\otimes i_{\lambda_2}\otimes \cdots \otimes i_{\lambda_s}\\
p_{\la}&=p_{\lambda_1}\otimes p_{\lambda_2}\otimes \cdots \otimes p_{\lambda_s} \\
q_{\la}&=q_{\lambda_1}\otimes q_{\lambda_2}\otimes \cdots \otimes q_{\lambda_s} .
\end{align*}
We define the \textbf{quantum Schur functor} $S_\lambda$ as the image of the composition of the following morphiphs
$$
\xymatrix{
\bw_q^\la \ar[r]^{\alpha_\la} & \bigotimes^{d} \ar[r]^{T_{\sigma_\la}} & \bigotimes^{d} \ar[r]^{q_{\la'}} & S_q^{\la'},
}
$$
Define the \textbf{quantum Weyl functor}  $W_{\lambda}$ as the image of the composition of the following morphisms:
$$
\xymatrix{
\Gamma_q^{\la} \ar[r]^{i_{\la}} & \bigotimes^{d} \ar[r]^{T_{\sigma_{\la}}} & \bigotimes^{d} \ar[r]^{p_{\la'}} & \bw_q^{\la'}.
}
$$

For any partition $\lambda$,  $S^\lambda$ and $W_\lambda$ are well-defined objects in $\PP_q$. This uses that for any $n$, $S^\lambda(n)$ and $W_\lambda(n)$ are free $\kk$-module of finite rank, which we have by the remarkable results of Hashimoto and Hayashi on the freeness of quantum Schur and Weyl modules \cite[Theorem 6.19, Theorem 6.23]{HH}.

\begin{theorem}
\label{SW-functor}
For any partition $\lambda$,  we have a canonical isomorphism 
$$W_{\lambda'}\simeq  (S_{\lambda})^\sharp .$$
\end{theorem}
\begin{proof}
We first note that $\sigma_{\lambda'}=(\sigma_\lambda)^{-1}$.
Then the theorem follows from Proposition \ref{Braid_Duality}, \ref{symmetrization_projection}.
\end{proof}

Suppose that $\ell(\la)\leq n$.  By work of Hashimoto and Hayashi $S_\la(n)$ is the Schur module and $W_\la(n)$ is the Weyl module of the $q$-Schur algebra $S_q(n,n;d)$ (cf. Definition 6.7, Theorem 6.19, and Definition 6.21 \cite{HH}).
Let $L_{\la}$ be the socle of the functor $S_{\la'}$.  Recall that this is the maximal semisimple subfunctor of $S_{\la'}$.   

\begin{proposition}
\label{thm: F_la}
The functors $L_\la$, where $\la$ ranges over all partitions of $d$, form a complete set of representatives for the isomorphism classes of irreducible objects in $\PP_q^d$.
%\item When $q$ is generic, i.e. $q$ is not a nontrivial root of unity, we have $F_{\la} \cong K_\la \cong L_{\la'}$.
\end{proposition}
\begin{proof}
By Theorem \ref{Rep_thm}, $\PP_q^d \cong \Mod{S_q(n,n;d)}$ for any $n\geq d$.  To prove the  statement it suffices to show that $\{L_\la(n)\}$ form a complete set of representatives for irreducible $S_q(n,n;d)$-modules.  This follows from Lemma 8.3 and Proposition 8.4 in \cite{HH}.    
\end{proof}

%We define the \textbf{quantum Schur functor} $L_\la\in\PP_q^d$ as the image of $\phi_\la$: 
%$$L_\la=Im(\phi_\la).$$  
%The evaluation of this functor results in the \textbf{Schur module} $L_\la(n)=Im(\phi_\la(n))$ of $S_q(n,d)$ \cite[Def. 6.7]{HH}.
%
%The space $\bigwedge^\lambda_q(n)$ has a standard basis indexed by $RST(\lambda,n)$, the set of row-standard tableau of shape $\lambda$ with entries from $1,...,n$.  Here row-standard means that the entries strictly increase from left to right along any row.  
%We let $\ove_T$ denote the basis element of $\bigwedge^\lambda_q(n)$ corresponding to the tableau $T$ in $RST(\lambda,n)$.  For example, for $T={\tiny \young(23,1,2)}$ we have $\ove_T=e_2\wedge_q e_3 \tn e_1 \tn e_2 \in \bigwedge_q^{(2,1,1)}(n)$.
%
%The following summarizes the basic properties of Schur modules.
%
%\begin{proposition}
%\begin{enumerate}
%\item The set 
%$$
%\{\phi_\la(\ove_T):T\in ST(\la,n)\}
%$$
%is a basis of $L_\la(n)$.
%\item  The irreducible $S_q(n,d)$-module is 
%\item  If $q$ is generic (i.e. $q^r\neq1$ for any $r>1$) and $\ell(\la)\leq n$ then $L_\la(n)$ is isomorphic to the simple $S_q(n,d)$-module of highest weight $\la'$. 
%\end{enumerate}
%\end{proposition}
%
%\begin{proof}
%The first statement is Theorem 6.19 in \cite{HH}.  The second statement follows from  Lemma 8.7 and Corollary 8.10 in \textit{loc. cit.}
%\end{proof}
%
%

\section{Invariant theory of quantum general linear groups}
\label{sec-invariants}
In this section, we assume $\kk$ is algebraically closed and $q$ is a generic element in $\kk$.  Our aim is to show that the theory of quantum polynomial functors affords a streamlined derivation of the invariant theory of the quantum general linear groups $\sO_q({\rm GL}_n)$, with significantly simpler proofs.  Essentially, the proofs are immediate consequences of the representability theorem (Theorem \ref{Rep_thm}).  

 Recall that $\sO_q({\rm GL}_n)$ is the localization of $A_q(n,n)$ by the quantum determinant, 
$$
{\rm det}_q:=\sum_{\sigma \fS_n}(-q^{-1})^{\ell(\sigma)}x_{1\sigma(1)}\cdots x_{n\sigma(n)}.
$$  
$\sO_q({\rm GL}_n)$ is a Hopf algebra, and we denote its antipode by $\iota$.  For more details a good source is Chapter 5 of \cite{PW}.

Following Howe's approach to classical invariant theory (cf.\cite{Ho}), we first prove a quantum analog of $({\rm GL}_m,{\rm GL}_n)$ duality.  In the classical case Howe's proof is based on a geometric argument that the matrix space is spherical \cite{Ho} (although one can give also combinatorial proofs using the Cauchy decomposition formula\footnote{We thank the anonymous referee for pointing this out to us.}).  While this geometric argument fails in the quantum case, we show that Quantum $({\rm GL}_m,  {\rm GL}_n)$ duality is a direct consequence of the Theorem \ref{Rep_thm}.  We then show that, as in the classical case, quantum analogs of the first fundamental theorem and Schur-Weyl duality follow from Quantum $({\rm GL}_m,  {\rm GL}_n)$ duality.\\

By definition a representation of  $\sO_q({\rm GL}_n)$ is a right comodule $V$ of $\sO_q({\rm GL}_n)$.    A left module of the $q$-Schur algebra $S_q(n,n;d)$ is naturally a representation of $\sO_q({\rm GL}_n)$.  By analogy with the classical setting, any representation of $\sO_q({\rm GL}_n)$ coming from $S_q(n,n;d)$ is a polynomial representation  of degree $d$.

By Theorem \ref{thm: F_la} $L_\la(n)$ is an irreducible representation $\sO_q({\rm GL}_n)$ , and any irreducible representation of $\sO_q({\rm GL}_n)$ is isomorphic to $L_\la(n)$ for a unique $\la$ such that $\ell(\la)\leq n$. 

The comultiplication $\Delta:A_q(\ell,n)\to A_q(\ell,m)\otimes A_q(m,n)$
induces actions of the quantum general linear group by left and right multiplication on  quantum $m\times n$ matrices:
\begin{align*}
\mu_L':A_q(m,n)&\to\sO_q({\rm GL}_m) \otimes A_q(m,n)\\
\mu_R:A_q(m,n)&\to A_q(m,n)\otimes \sO_q({\rm GL}_n)
\end{align*} 
These maps commute and preserve degree.  We define 
$$\mu_L:=P\circ (\iota \otimes 1)\circ\mu_L'  :    A_q(m,n)  \to  A_q(m,n)\otimes \sO_q({\rm GL}_m),$$
 where $P$ is the flip map.  Then using $(\mu_L\otimes1)\circ\mu_R$, we regard $A_q(m,n)$ as a representation of $\sO_q({\rm GL}_m)\otimes \sO_q({\rm GL}_n)$. 

Given a representation $V$ of $\sO_q({\rm GL}_n)$ let $V^*$ be the contragredient represenation of $V$, i.e. twist the left coaction of $\sO_q({\rm GL}_n)$ on the dual space $V^*$ by the antipode $\iota$. 

\begin{theorem}[ Quantum $({\rm GL}_m,  {\rm GL}_n)$ duality]
As a representation of  $\sO_q({\rm GL}_m)\otimes \sO_q({\rm GL}_n)$ we have a multiplicity-free decomposition:
$$
A_q(m,n)_d \cong \bigoplus_{\la}L_\la(m)^*\otimes L_\la(n),
$$
where $\la$ runs over all partitions of $d$ such that $\ell(\la)\leq min(m,n)$.
\end{theorem}

\begin{proof}
By Theorem \ref{Rep_thm} the category $\PP_q^d$ is equivalent to the category $\Mod{S_q(n,n;d)}$. Hence the category $\PP^d_q$ is semi-simple, and the simple objects are, up to equivalence, the functors $L_\la$ where $\la$ ranges over partitions of $d$.  (Since $q$ is generic $L_\la\cong W_\la \cong S_{\la'}$.)  By Proposition \ref{Yonedalemma} for any $m\geq 0$ there  exists a natural isomorphism  
$\Hom_{\PP_q}(\Gamma_q^{d,m},L_\lambda)\simeq L_\lambda(m)$. Moreover, $L_\la(m)=0$ if $m>\ell(\la)$.  Hence we have the following decomposition
\begin{align*}
\Gamma_q^{d,m} & \cong \bigoplus_{\la}  L_\la\otimes \Hom_{\PP_q^d}(L_\la,\Gamma_q^{d,m}) \\
& \cong \bigoplus_{\la} L_\la\otimes  \Hom_{\PP_q^d}(   \Gamma_q^{d,m},L_\la)^*\\
& \cong \bigoplus_{\la} L_\la\otimes  L_\la(m)^* ,   
\end{align*}
where the second isomorphism follows from the natural pairing 
$$\Hom_{\PP_q^d}(L_\la,\Gamma_q^{d,m}) \times \Hom_{\PP_q^d}(   \Gamma_q^{d,m},L_\la)\to \Hom_{\PP_q^d}(L_\la,L_\la)\simeq \kk. $$
Evaluating both sides at $n$ yields
\begin{equation}
\label{Howedecomp}
\Hom_{\cH_d}(V_m^{\otimes d}, V_n^{\otimes d})\cong \bigoplus_{\la} L_\la(n)\otimes L_\la(m)^*. 
\end{equation}
This proves the theorem, since 
$$A_q(m,n) \cong (S_q(n,m;d))^* \cong (\Hom_{\cH_d}(V_n^{\otimes d}, V_m^{\otimes d}))^*\simeq \bigoplus_{\la}L_\la(m)^*\otimes L_\la(n), .$$

\end{proof}

In analogy with the classical setting, Quantum $({\rm GL}_m,  {\rm GL}_n)$ duality is equivalent to quantum FFT and Jimbo-Schur-Weyl duality.  We briefly mention these connections.

Given three numbers $\ell,m,n$  define a representation of $\sO_q({\rm GL}_m)$ on 
$A_q(n,m)\otimes A_q(m,\ell)$ as follows:
\begin{align*}
\xymatrix{
A_q(n,m)\otimes A_q(m,\ell)\ar[r]^>>>>>{\mu_R\otimes\mu_L} & A_q(n,m)\otimes \sO_q({\rm GL}_m)\otimes A_q(m,\ell) \otimes \sO_q({\rm GL}_m)  \ar[d] \\
&  A_q(n,m)\otimes A_q(m,\ell)\otimes \sO_q({\rm GL}_m)
}
\end{align*}
In the above diagram, the  downward map is given by multiplication in $\sO_q({\rm GL}_m)$.

Recall that given a right comodule $a: V\to V\otimes A$ of a Hopf algebra $A$,  the space of $A$-invariants in $V$ is a subspace of $V$
$$ V^{A}:=\{  v\in V| a(v)=v\otimes 1     \}$$

\begin{theorem}[Quantum FFT]
\label{FFT}
For any $\ell,m,n$ the image of the comultiplication $$\Delta:A_q(n,\ell)\to A_q(n,m)\otimes A_q(m,\ell)$$ lies in the subspace of $\sO_q({\rm GL}_m)$-invariants, and, moreover, gives rise to a surjective map
$$
A_q(n,\ell)\to ( A_q(n,m)\otimes A_q(m,\ell))^{\sO_q({\rm GL}_m)}.
$$
\end{theorem}

\begin{proof}
First we note that for any representation $V$ of $\sO_q({\rm GL}_m)$, by complete reducibility, we have $(V^*)^{\sO_q({\rm GL}_m)}\simeq (V^{\sO_q({\rm GL}_m)})^*$. Then taking duals, by Proposition \ref{Natural_Iso},
it suffices to show that 
the following map is injective:
\begin{equation}
\label{Dual_FFT}
(\Hom_{\cH_d}(V_\ell^{\otimes d}, V_{m}^{\otimes d})\otimes \Hom_{\cH_d}(V_m^{\otimes d},V_n^{\otimes d}))^{\sO_q({\rm GL}_m)}\to \Hom_{\cH_d}(V_\ell^{\otimes d}, V_n^{\otimes d}) ,
\end{equation}
where $\sO_q({\rm GL}_m)$ acts diagonally on the left hand side.  This follows immediately from $(\sO_q({\rm GL}_m),\sO_q({\rm GL}_n))$ duality, since by Equation \ref{Howedecomp} the above map is precisely the inclusion
$$
\bigoplus_{\ell(\la)\leq\ell,m,n}L_\la(\ell)^*\otimes L_\la(n) \to \bigoplus_{\ell(\la)\leq\ell,n}L_\la(\ell)^*\otimes L_\la(n).
$$
\end{proof}

Finally, consider tensor space $V_m^{\tn d}$.  As a representation of $\sO_q({\rm GL}_m)$ we have a decomposition
\begin{equation}
\label{Schur_Weyl}
V_m^{\otimes d}\cong \bigoplus_{\lambda} L_\lambda(m)\otimes M_{\lambda}
\end{equation}
where the $\lambda$ runs over all partitions of $d$, and $M_\lambda=\Hom_{\sO_q({\rm GL}_m)}(L_\lambda(m), V_m^{\otimes d})$.  Notice that by the construciton of $L_\la$, we have that $L_\la(m)=0$ if $\ell(\la)>m$.  Hence the sum above is over all partitions $\la$ of $d$ such that $\ell(\la)\leq m$.  Note also that $M_\la$ are naturally $\cH_d$-modules.

\begin{theorem}[Jimbo-Schur-Weyl duality]
Equation (\ref{Schur_Weyl}) is a multiplicity-free decomposition of $V_m^{\tn d}$ as an $\sO_q({\rm GL}_m)\times\cH_d$-representation.  In particular,  $M_\la$ are irreducible pairwise inequivalent $\cH_d$-modules.
\end{theorem}

\begin{proof}
We will deduce this result from the quantum FFT.  Indeed, applying  Theorem (\ref{FFT}) to the case $n=m=\ell$,   it follows that for any partition $\lambda$ of $d$ such that $\ell(\la)\leq m$, the following map is injective:
$$\bigoplus_{\mu } \Hom_{\cH_d}(M_{\lambda}, M_{\mu})\otimes \Hom_{\cH_d}(M_\mu, M_{\lambda})\to \Hom_{\cH_d}(M_{\lambda},M_\lambda) , $$
where $\mu$ runs over all partition of $d$ with $\ell(\mu)\leq m$. This implies that $M_\lambda$ is irreducible as $\cH_d$-module and for any $\lambda\not= \mu$, $M_\lambda$ and $M_\mu$ are non-isomorphic, proving the result.
\end{proof}

\begin{remark}
\label{LastRemark}
\begin{enumerate}
\item[]
\item One can easily show that Jimbo-Schur-Weyl duality implies $(\sO_q({\rm GL}_m),\sO_q({\rm GL}_m))$ duality using Proposition \ref{Natural_Iso}.  This completes the chain of equivalences, and hence the three basic theorems of quantum invariant theory $(\sO_q({\rm GL}_m),\sO_q({\rm GL}_m))$  duality, the quantum FFT, and Jimbo-Schur-Weyl duality) are all equivalent, as in the classical case done by Howe\cite{Ho}.
\item Recall our standing assumption that $q$ is generic and $\kk$ is algebraically closed.  The approach taken here essentially uses only the fact that the functors $\Gamma_q^{d,n}$ are projective generators for $n\geq d$, and this will work in any other setting of polynomial functors which has an analogous property, namely the classical and super cases \cite{FS,Ax}.  Note that in the super-case, although we don't have semisimplicity of representations in general, the tensor powers of the standard representation of $\mathfrak{gl}_{m|n}$ are semisimple \cite{BR} and so the methods here do carry over to the super case.  Therefore this approach can be used to give a new and uniform development for the classical, quantum and super invariant theories of the general linear group.
\end{enumerate}
\end{remark}

\section{Obstructions to quantum plethysm}
\label{sec:final}
Composition of quantum polynomial functors, which would provide a sought-after theory of quantum plethysm, is  absent from our theory.  
In this final section we discuss why this is the case, and further speculate on possible generalisations of our construction to allow for composition.  For convenience, we assume $\kk$ is a field. 

First we recall how composition works in the classical setting of Section \ref{sec:classical}.  Let $F\in\PP^d$ and $G\in\PP^e$.  Then $F\circ G\in\PP^{de}$ is given as follows: On objects $F\circ G(V)=F(G(V))$ and for spaces $V,W$ we define $\Hom_{\Gamma^{de} \Vect}(V,W)\to \Hom(FG(V),FG(W))$ in steps.  First consider 
$$
\Hom_{\Gamma^e\Vect}(V,W)\to \Hom(G(V),G(W)).
$$
Apply the functor $\Gamma^d$ to this linear map to obtain
$$
\Gamma^d(\Hom_{\Gamma^e\Vect}(V,W)) \to \Gamma^d(\Hom(G(V),G(W))).
$$
Note that for any space $X$ we have $\Gamma^{de}(X)\subset\Gamma^d(\Gamma^e(X))$, which is compatible with the standard embedding $\fS_d\times\fS_e \subset \fS_{de}$.  Therefore we have $\Hom_{\Gamma^{de}\Vect}(V,W)) \to \Gamma^d(\Hom(G(V),G(W)))$, which we compose with 
$$
\Gamma^d(\Hom(G(V),G(W))) \to \Hom(FG(V),FG(W))
$$
to obtain the desired map.

This construction does not generalize to the quantum setting for several reasons.  We focus on the most basic one, namely that we can't make sense of $F(G(n))$ in our construction since $G(n)$ is not an object in the quantum divided power category.  Of course we really think of $n$ as the standard Yang-Baxter space $(V_n,R_n)$, and so we should restate this problem by saying that $G(n)$ is not a Yang-Baxter space, let alone a standard one.  This suggests that we should enlarge the set of objects of the quantum divided power category.

More precisely, let $\cY$ be the category of all Yang-Baxter spaces (how we define morphisms is not important for the purposes of this discussion), and let $\cY_{st}$ be the subcategory of standard Yang-Baxter spaces.  We would like an intermediate category $\cY_{st}\subset\sC\subset\cY$ to use as the objects of the quantum divided power category $\Gamma^d\sC$.  Then we would like representations $F:\Gamma^d\sC \to \V$ to satisfy the property that for $V\in\sC$ we have $F(V)\in\sC$, allowing us to make sense of $F(G(V))$ for two such functors $F,G$.

Let's suppose such a category $\sC$ exists and try to determine some of its properties.  Perhaps the most basic quantum polynomial  functor we seek is the tensor product functor.  It turns out that a notion of tensor product is relatively easy to construct.  Indeed given a Yang-Baxter space $(V,R)\in\cY$ and any $d>0$ define $w_d\in\fS_{2d}$  by
$$
w_d(i)=\begin{cases}
i+d \text{ if } i\leq d,\\
i-d \text{ if } i>d.
\end{cases}
$$
Then it's straight-forward to verify that $T_{w_d}:V^{\otimes 2d}\to V^{\otimes 2d}$ is a Yang-Baxter operator  and hence we can define $(V,R)^{\otimes d}=(V^{\otimes d},T_{w_d})\in\cY$.

Therefore we require that $\sC$ contains, along with all the standard Yang-Baxter spaces, their tensor products $(V_n^{\tn d},T_{w_d})$.  Note that this tensor product is consistent in the sense that $(V_n^{\tn d},T_{w_d})^{\tn e}=(V_n^{\tn de},T_{w_{de}})$.  

%because a key point is that for any $V\in\Gamma^d$, $G(V)^{\otimes d}$ is again an $\fS_d$-modules, and hence we can consider $\Gamma^d(\Hom(G(V),G(W)))$.  In contrast, if $G\in\PP_q^d$ then $G(n)^{\otimes d}$ is not naturally a $\cH_d$-module since $G(n)$ has no preferred basis.  In particular, we cannot make sense of $\Gamma^d_q(\Hom(G(m),G(n))$.
%
%It's a very interesting open problem to modify the definition of $\PP_q$ to get around this issue, and thus define a notion of composition.  One natural approach would be to enlarge the category $\Gamma_q^d$ to include Yang-Baxter pairs $(V,R)$ such that if $(V,R)\in \Gamma_q^d$ then there is also a notion of tensor powers, symmetric powers, etc.  of $(V,R)$.    
%
%Firstly note that a notion of tensor product is relatively easy to construct.  Indeed given a Yang-Baxter pairs $(V,R)$ and any $d>0$ define $w_0\in\fS_{2d}$  by
%$$
%w_0(i)=\begin{cases}
%i+d \text{ if } i\leq d,\\
%i-d \text{ if } i>d.
%\end{cases}
%$$
%Then it's straight-forward to verify that $T_{w_0}:V^{\otimes 2d}\to V^{\otimes 2d}$ is a Yang-Baxter operator  and hence $(V^{\otimes d},T_{w_0})$ is a Yang-Baxter pair.  

Next we would like to define analogs of symmetric and exterior powers.  We will see that this becomes very subtle, and for this we focus on symmetric and exterior squares.  

Classically we of course have $\bigotimes^2 \cong S^2\oplus \bw^2$.  This decomposition is closely related to the fact that for any $V\in\V$ the spectrum of the flip operator $V\otimes V\to V\otimes V$ given by $v\otimes w \mapsto w\otimes v$ has spectrum $\pm1$, as long as $\dim(V)\geq2$.  In our quantum setting we also have $\bigotimes^2 \cong S_q^2\oplus \bw_q^2$ since the spectrum of $R_n$ is $\{q,-q^{-1}\}$ for $n\geq 2$. 

The corresponding spectrum for Yang-Baxter spaces in $\sC$ is much more complicated.  Indeed, consider the following table, which we computed with the help of a computer:
\begin{table}[ht] 
%\caption{Scribe Assignments} % title of Table 
\centering % used for centering table 
\begin{tabular}{c ccc} % centered columns (4 columns) 
\hline\hline %inserts double horizontal lines 
YB space $V$ & Spectrum of YB operator on $V\otimes V$ \\ [0.5ex] % inserts table 
%heading 
\hline % inserts single horizontal line 
$V_2^{\tn 2}$ & $q^{-2}\to 1, -1\to 3, -q^2 \to 3, q^2 \to 4, q^4 \to 5$ \\
$V_3^{\tn 2}$ & $-q^{-2}\to 3, q^{-2}\to 9, -1\to 18, -q^2 \to 15, q^2 \to 21, q^4 \to 15 $ \\ 
$V_4^{\tn 2}$ & $q^{-4} \to 1, -q^{-2}\to15, q^{-2}\to35, -1\to60,-q^2\to45,q^2\to65,q^4\to35$ \\
$V_2^{\tn 3}$ & $-q^{-3}\to 1, q^{-1} \to 3, -q^2\to6, q^2\to6, -q^3\to8, q^3\to1,-q^5\to3$ \\
& $q^5\to9, -q^6\to10, q^6\to10, q^9\to 7$ \\
 [1ex] % [1ex] adds vertical space 
\hline %inserts single line 
\end{tabular} 
\label{table:nonlin} % is used to refer this table in the text 
\end{table}

In the right column, we use the notation ``eigenvalue $\to$ multiplicity'', so for instance the Yang-Baxter operator on $V_3^{\tn 2}\otimes V_3^{\tn 2}$ has eigenvalue $-q^{-2}$ with multiplicity $3$. We see  that the spectrum of the Yang-Baxter operators on tensor squares of objects in $\sC$ does not necessarily stabilize as the dimension of the Yang-Baxter space gets big.  (Although one might speculate that if $d$ is fixed and we let $n \to \infty$ then the spectrum of the Yang-Baxter operator of $V_n^{\otimes d}$ does stabilize.) 

This  suggests that instead of just decomposing $\bigotimes^2$ into a symmetric and exterior square, we should have an infinite decomposition 
$$
\xymatrix{
\bigotimes^2=\bigoplus_{\pm,n\in\bZ}F_{\pm,n},}
$$
where $F_{\pm,n}:\Gamma^2\sC \to \V$ is given by $F_{\pm,n}(V,R)=\pm q^n$-\text{eigenspace of }$R$.  

If true, a consequence is that in order for composition to be defined, for every $(V,R)\in\sC$ we must ensure that the Yang-Baxter spaces $F_{\pm,n}(V,R)$ belong to $\sC$.  Hence also the tensor powers of $F_{\pm,n}(V,R)$ must belong to $\sC$, as well as the compositions $F_{\pm,n}\circ F_{\pm,m}(V,R)$, etc.  
It appears that the resulting theory, if it can be constructed, will be much wilder than the quantum polynomial functors considered here.  (This is perhaps not surprising as the representation theory of the braid group is known to be extremely complicated.)  One must study fundamentally different phenomenon, which are no doubt interesting but pose significant challenges.  We hope the ideas put forth here are a significant first step in this story.


\begin{thebibliography}{99}
\bibitem[Ax]{Ax}
J.\,Axtell.  Spin polynomial functors and representations of Schur superalgebras.
Represent. Theory 17 (2013), 584-609.
\bibitem[BR]{BR} A.\,Berele, A.\,Regev. Hook young diagrams with applications to combinatorics and to representations of Lie superalgebras. Advances in Mathematics 64 (1987) 2, 118-175.
\bibitem[B]{B} G.\,Bergman. The diamond lemma for ring theory, Advances in Mathematics 29 (1978) 2, 178-218.
\bibitem[BDK]{BDK}J.\,Brundan, R.\,Dipper and A.\,Kleshchev.  Quantum linear groups and representations of ${\rm GL}_n({\mathbb{F}}_q)$. Mem.
\bibitem[BZ]{BZ} A. Berenstein, S. Zwicknagl. Braided symmetric and exterior algebras, Trans. Amer. Math. Soc. (360) 2008, no. 7, 3429-3472.
\bibitem[Bu]{Bu}T.\,B\"uhler. Exact categories. Expo. Math., 28(1):1-69, 2010.
\bibitem[DJ1]{DJ1} R.\,Dipper, G.\,James, The q-Schur algebra, Proc. London Math. Soc. (3) 59
(1989), 23-50.
\bibitem[DJ2]{DJ2}R.\,Dipper, S.\,Donkin, Quantum GLn, Proc. London Math. Soc. (3) 63 (1991), 165-211.
\bibitem[FS]{FS}E.\,M.\,Friedlander and A.\,Suslin. Cohomology of finite
group schemes over a field.  Invent. Math. 127 (1997), no. 2, 209-270.
Amer. Math. Soc. 149 (2001), no. 706, viii+112 pp.
\bibitem[GLR]{GLR} K.\,Goodearl,  T.\,Lenagan and L.\,Rigal.  The first fundamental theorem of coinvariant theory for the quantum general linear group. Publ. Res. Inst. Math. Sci. 36 (2000), no. 2, 269-296. 
\bibitem[Ho]{Ho}R.\,Howe.  Perspectives on invariant theory: Schur duality, multiplicity-free actions
and beyond. The Schur lectures (1992) (Tel Aviv), 1-182, Israel
Math. Conf. Proc., 8, Bar-Ilan Univ., Ramat Gan, 1995. 
\bibitem[HY1]{HY1}J.\,Hong and O.\,Yacobi.   Polynomial Representations of general linear groups and Categorifications of Fock Space. Algebras and Representation Theory. October 2013, Volume 16, Issue 5, pp 1273-1311 
\bibitem[HY2]{HY2}J.\,Hong and O.\,Yacobi.  Polynomial functors and categorifications of Fock space II,   Advances in Math. Volume 237, 360-403 (2013).
\bibitem[HTY]{HTY}J.\,Hong, A.\,Touz\'e and O.\,Yacobi. Polynomial functors and categorifications of Fock space, Symmetry: Representation Theory and its Applications in honor of Nolan Wallach, Progress in Mathematics, Birkauser, Vol.257 (2014).  
\bibitem[HH]{HH} M.\,Hashimoto and  T.\,Hayashi. Quantum multilinear algebra. Tohoku Math. J. (2) 44 (1992), no. 4, 471-521.
\bibitem[Kr]{Kr} H.\,Krause.  Koszul, Ringel and Serre duality for strict polynomial functors.  
Compos. Math. 149 (2013), no. 6, 996-1018.
\bibitem[Ku]{Ku} N.\,J.\,Kuhn. Rational cohomology and cohomological stability in generic representation theory, Amer. J. Math. 120 (1998), 1317-1341.
\bibitem[LZZ]{LZZ} G.\,I.\,Lehrer, H.\,Zhang and R.\,B.\, Zhang. A quantum analogue of the first fundamental theorem of classical invariant theory. Comm. Math. Phys. 301 (2011), no. 1, 131-174.
\bibitem[Ma]{Ma}Y.\,Manin.  Quantum groups and algebraic groups in noncommutative geometry. Quantum groups and their applications in physics (Varenna, 1994), 347?358, Proc. Internat. School Phys. Enrico Fermi, 127, IOS, Amsterdam, 1996. 
\bibitem[Ph]{Ph} H.\,H.\,Ph\'ung. Realizations of quantum hom-spaces, invariant theory, and quantum determinantal ideals. J. Algebra 248 (2002), no. 1, 50-84.
\bibitem[PW]{PW}B.\,Parshall and J.\,P.\,Wang.  Quantum linear groups. Mem.
Amer. Math. Soc. 89 (1991), no. 439, vi+157 pp.
\bibitem[T]{T} M.\,Takeuchi. A short course on quantum matrices, New Directions
in Hopf Algebras, MSRI Publications, Vol. 43, 2002.
\bibitem[Zh]{Zh}R.\,B.\,Zhang. Howe duality and the quantum general linear group. (English summary) 
Proc. Amer. Math. Soc. 131 (2003), no. 9, 2681-2692 (electronic). 
\end{thebibliography}
\end{document}